\documentclass[a4paper,fleqn, 12pt]{cas-sc}

\usepackage[numbers,sort&compress]{natbib}
\def\tsc#1{\csdef{#1}{\textsc{\lowercase{#1}}\xspace}}
\tsc{WGM}
\tsc{QE}
\tsc{EP}
\tsc{PMS}
\tsc{BEC}
\tsc{DE}
\usepackage{framed,multirow}
\usepackage{float} 
\usepackage{hyperref}

\usepackage{mathtools}
\usepackage{amssymb}
\usepackage{amsmath,amsthm}
\usepackage{mathrsfs}
\usepackage{bm} 
\newtheorem{theorem}{Theorem}
\newtheorem{lemma}{Lemma}

\newtheorem{assumption}{Assumption}
\usepackage{mleftright}
\mleftright 
\newcommand{\mathd}{\mathrm{d}}
\newcommand{\diff}{\,\mathrm{d}}

\newcommand{\LaF}{\widehat{\mathscr{L}}}
\newcommand{\LaP}{{\mathscr{L}}}
\usepackage{url}
\usepackage{xcolor}
\usepackage{subcaption}
\usepackage{float}
\definecolor{newcolor}{rgb}{.8,.349,.1}
\usepackage{algorithm}      
\usepackage{algpseudocode}  

\numberwithin{table}{section}
\begin{document}
\let\WriteBookmarks\relax
    \def\floatpagepagefraction{1}
    \def\textpagefraction{.001}
\shorttitle{Laguerre Minimum Action Method}
\shortauthors{S. Huang, Y. Yue and H. Yu}


\title[mode = title]{
An Efficient Laguerre Minimum Action Method for Computing Quasi-Potentials
}

\author[U,I,C]{Shenghe Huang}
\fnmark[1]
\ead{shenghe.huang@lsec.cc.ac.cn}

\author[U,I]{Yishuang Yue}
\fnmark[1]
\ead{yueyishuang21@mails.ucas.ac.cn}

\fntext[fn1]{These authors contributed equally to this work.}
\author[I,U]{Haijun Yu}[orcid=0000-0002-5742-0327]
\cormark[1] 
\cortext[cor1]{Corresponding author. \ead{hyu@lsec.cc.ac.cn}}

\affiliation[I]{organization={State Key Laboratory of Mathematical Sciences (SKLMS) \& LSEC, Institute of Computational Mathematics and Scientific/Engineering Computing, Academy of Mathematics and Systems Science},
                city={Beijing},
                postcode={100190},
                country={China}}

\affiliation[U]{organization={School of Mathematical Sciences, University of Chinese Academy of Sciences},
                city={Beijing},
                postcode={10049},
                country={China}}
\affiliation[C]{organization={Present Address: CSSC SEAGO System Technology Co., Ltd},
                city={Shanghai},
                postcode={200010},
                country={China}}

\begin{abstract}
Minimum action methods provide a powerful framework for analyzing rare transitions in small-noise-driven dynamical systems, but their practical performance is often limited by time truncation and
parameter sensitivity in infinite-horizon problems.
In this paper, we develop an efficient Laguerre spectral minimum action method (LMAM) for computing quasi-potentials associated with fixed points of dynamical systems.
Based on the large deviation framework, the method computes minimum action paths by
formulating the problem on a semi-infinite time interval and discretize the temporal direction using Laguerre functions.
An appropriate time rescaling strategy is proposed to enhance accuracy and convergence of the Laguerre spectral approximation.
To efficiently handle nonlinear terms, we employ an improved procedure for evaluating Laguerre--Gauss--Radau quadrature, which enables stable and accurate double-precision computations with a large number of Laguerre modes.
Precise numerical analysis for the linear problem and a local result for the nonlinear case are developed. Numerical experiments including both ordinary and partial differential equations (Allen-Cahn and Navier-Stokes) are presented to illustrate the accuracy and efficiency of the proposed method.
\end{abstract}




\maketitle

\section{Introduction}

Many physical, chemical, and biological systems can be modeled as dynamical systems subject to small random perturbations. Although the amplitude of such stochastic forcing is typically weak, it may induce rare but crucial events over long-time evolution. Typical examples include nucleation processes in chemical reactions, transitions between multiple stable equilibria, escape from stable limit cycles, and instability phenomena in fluid dynamics~\cite{DellaBCC1998TransitionPath, BolhuCDG2002TransitionPath,e2005transition}. Understanding the mechanisms of these rare events, computing the most probable transition paths
and estimating their occurrence probabilities constitute central problems in statistical physics and nonlinear dynamics~\cite{EV2010TransitionpathTheory, GrafkSV2024SharpAsymptotic}.

Freidlin--Wentzell (F--W) large deviation theory \cite{freidlin_random_2012}
provides a rigorous mathematical framework for studying such phenomena. According to this theory, in the limit where the noise intensity $\varepsilon \to 0$, the probability that a stochastic trajectory deviates from the deterministic path is governed by the Freidlin--Wentzell action functional. The minimizer of this functional corresponds to the minimum action path (MAP), and the associated minimum value is referred to as the \emph{quasi-potential}. The quasi-potential not only quantifies the ``cost'' required for a noise-driven system to cross a potential barrier, but also plays a fundamental role in characterizing the stability of non-gradient systems, nonequilibrium stationary distributions, and the geometric structure of phase space.

Despite its importance, the numerical computation of the quasi-potential is highly challenging. 

For gradient systems, a transition path always goes through some saddle point and take the negative direction of the corresponding dynamical trajectory. So the key issue of gradient systems is to calculate the saddle points, for which a lot of efficient numerical methods have been built, including the dimer method (see, e.g. \cite{HenkeJ1999DimerMethod, ZhangDu2012ShrinkingDimer}), the nudged elastic band method~\cite{JonssMJ1998NudgedElastic,HenkeJ2000ImprovedTangent}, the string method~\cite{ERV2002StringMethod, ERV2007SimplifiedImproved, DuZhang2009ConstrainedString, ChengLEZS2010NucleationOrdered,RenV2013ClimbingString}, the gentlest acsend dynamic method~\cite{EZhou2011GentlestAscent, LiLuYang2015GentlestAscent,GuZhou2018SimplifiedGentlest}, and eigenvector-following approach~\cite{GaoLZ2015IterativeMinimization,GaoLZ2016IterativeMinimization, YinWCZZ2020ConstructionPathway, YinJSZZ2021TransitionPathways, SuWZZZ2025ImprovedHighIndex}.

For non-gradient systems, the value of quasi-potential depends on the entire MAP. One have to solve the minimization problem derived from F-W large deviation theory. 
One major difficulty arises from the fact that the  optimization problem is defined over an infinite time interval ($t \in (-\infty, T]$). As a result, traditional numerical methods must carefully balance accuracy and efficiency. The original minimum action method (MAM~\cite{e_minimum_2004}) minimizes the action functional over a truncated finite time interval, which requires prescribing the integration time in advance. However, when the optimal transition time tends to infinity, this approach often suffers from severe resolution issues near critical points. To improve robustness with respect to time parametrization and mesh resolution, several variants have been developed. The adaptive minimum action method (aMAM~\cite{zhou2008adaptive})  introduces a moving-mesh reparametrization to redistribute grid points along the path and mitigate clustering near equilibria. The geometric minimum action method (gMAM \cite{heymann_geometric_2008, VandeH2008GeometricMinimum,grafke2017long}) further reformulates the problem as a minimization of the geometric action on the space of curves, thereby removing explicit dependence on the time parametrization and providing a unified treatment of finite- and infinite-time transitions. However, gMAM changes the form of the action functional, for which, it is not easy to design high-order numerical discretizations. On contrary, Wan et al. introduce a  linear time rescaling in MAM (tMAM~\cite{wan_minimum_2015,wan_convergence_2018}), allowing high-order discretization (hp-MAM~ \cite{wan2018hp}), meanwhile handling both finite- and infinite-time transition paths within a unified framework. 
To address situations in which transition paths involve limit cycles and exhibit extremely long transient behavior, a minimum action method based on local quadratic approximation near the limit cycle was proposed by Lin et al.\cite{lin_quasi-potential_2019}. More recently, inspired by the rapid development of machine learning, a neural-network-based geometric minimum action method, termed deep gMAM, was proposed to address rare-event computations in high-dimensional and spatially extended nonequilibrium systems~\cite{simonnet2023computing}. Its main feature is that the transition path is represented globally by a neural network rather than by explicit local mesh discretization, leading to a flexible mesh-free framework that is particularly suitable for complex path geometries.

Although these developments have significantly improved the numerical treatment of minimum action problems, key challenges remain. 
Most existing methods---including many variants of MAM---still rely on local discretizations such as finite differences, finite elements, or piecewise polynomial approximations, and thus typically deliver only algebraic-order accuracy (see e.g. \cite{wan_convergence_2018,hong2025convergence}). 
As a consequence, resolving long transition paths and slowly varying trajectory segments often requires a large number of degrees of freedom. 
These difficulties become even more pronounced for spatially extended systems, where spatial discretization leads to massive systems and requires tremendous computational power. 
Moreover, even recent mesh-free or learning-based approaches are generally not tailored to exploit the intrinsic semi-infinite-time structure of transition paths issuing from fixed points.

Motivated by these observations, we propose an efficient \emph{Laguerre spectral minimum action method} (LMAM) for computing quasi-potentials and transition paths associated with fixed points of dynamical systems. Our approach employs Laguerre functions as basis functions, which are $L^2$-orthogonal bases for the semi-infinite interval $[0,\infty)$ and therefore naturally suited for describing infinite-time trajectories originating from fixed points. Laguerre spectral method has been used for solving partial differential equations on unbounded domains (see e.g.~\cite{Mavriplis1989LaguerrePolynomials, CoulaFK1990LaguerreSpectral,GuoShen2000LaguerreGalerkinMethod, Shen2000StableEfficient,BoydRB2003PseudospectralMethods, BaoShen2005FourthorderTimesplitting,AzaieSXZ2009LaguerreLegendre,ShenWangYu2016EfficientSpectralelement} ), and time evolutionary problems (see .e.g~\cite{Wahlberg1991SystemIdentification, Mikha1999SpectralLaguerre,Ben-yZ2007NumericalIntegration,GuoWTW2008IntegrationProcesses}), but has not been used for solving transition path problems.

The main contributions of this paper are summarized as follows:
\begin{enumerate}
	\item \textbf{Efficient spectral discretization.}  
	We develop a Galerkin spectral method based on Laguerre functions to discretize the variational problem, transforming it into a system of algebraic equations with spectral accuracy. This approach allows the essential features of infinite-time trajectories to be captured using a relatively small number of basis functions.
	
	\item \textbf{Adaptive scaling factor strategy.}  
	An adaptive strategy for selecting the scaling factor in Laguerre approximation is introduced, which dynamically balances the truncation errors in temporal domain and frequency domain, according to recent theoretical results~\cite{hu2026scaling,huscaling}. This significantly enhances the robustness of the method when applied to systems with disparate characteristic time scales.
	
	\item \textbf{High-accuracy treatment of nonlinear terms.}  
	To address the numerical instability arising from the evaluation of high-order Laguerre functions in double-precision arithmetic, we employ an improved Gauss--Radau quadrature procedure we developed recently~\cite{huang2024improved}. This enables accurate and stable large-scale computations of the nonlinear terms using a pseudo-spectral approach.
	
	\item \textbf{Validation on high-dimensional physical models.}  
	The LMAM is applied to classical benchmark problems, including the Allen--Cahn equation and the two-dimensional Navier--Stokes equations. Numerical experiments demonstrate the high efficiency and accuracy of the proposed method for spatially extended systems.
\end{enumerate}

The remainder of the paper is organized as follows.
Section~\ref{Problem Formulation} reviews the Freidlin--Wentzell large deviation principle and the variational formulation of the minimum action method.
Section~\ref{Numerical Strategies} derives the corresponding Euler--Lagrange (E--L) equation and summarizes the main numerical paradigms.
Section~\ref{Laguerre Spectral Method} introduces the Laguerre spectral discretization and the numerical treatment of both linear and nonlinear terms, and discusses the strategy for selecting the time-scaling factor.
Section~\ref{Spectral Properties} presents the convergence analysis, including the theoretical optimal scaling factor for linear systems and discrete error estimates for nonlinear systems.
Section~\ref{Numerical Experiments} demonstrates the performance of the proposed method on a range of benchmark problems, including linear and nonlinear systems, the Allen--Cahn equation, and the Navier--Stokes equations.
Finally, Section~\ref{conclusion} concludes the paper.

\section{Problem formulation}
\label{Problem Formulation}

We consider dynamical systems subject to small stochastic perturbations, whose evolution is governed by the following It\^o stochastic differential equation (SDE):
\begin{equation}
	\mathrm{d} X^{\varepsilon}
	=
	b(X^{\varepsilon})\,\mathrm{d} t
	+
	\varepsilon\, \sigma(X^{\varepsilon})\,\mathrm{d} W,
	\qquad
	X^{\varepsilon}(t) \in \mathbb{R}^m,
	\label{eq:SDE}
\end{equation}
where $W(t)$ denotes an $n$-dimensional Wiener process, $\sigma(x)$ is a matrix-valued function, and the superscript $\varepsilon$ indicates dependence on the noise intensity parameter $\varepsilon$. Formally, \eqref{eq:SDE} can be written as
\[
\frac{\mathrm{d} X^{\varepsilon}}{\mathrm{d} t}
=
b(X^{\varepsilon})
+
\varepsilon\, \sigma(X^{\varepsilon})
\frac{\mathrm{d} W}{\mathrm{d} t}.
\]
In the limit $\varepsilon \to 0$, the stochastic system converges to the deterministic dynamical system $\dot{x} = b(x)$.

According to Freidlin--Wentzell large deviation theory, over a finite time interval $[t_0, t_0+T]$, the probability that the stochastic trajectory $X_t^{\varepsilon}$ deviates from the deterministic trajectory and follows a prescribed path $\varphi(t)$ is characterized by the action functional $S_T[\varphi]$:
\begin{equation}\label{eq:ActionSt}
    S_T[\varphi]
=
\frac{1}{2}
\int_{t_0}^{t_0+T}\!\!
\sum_{i,j}
a_{ij}(\varphi)
\bigl(\dot{\varphi}^i - b^i(\varphi)\bigr)
\bigl(\dot{\varphi}^j - b^j(\varphi)\bigr)
\,\mathrm{d}t,
\quad
\varphi(t_0)=x_0,\ \varphi(t_0+T)=x_1,
\end{equation}
where $a_{ij} = (\Sigma^{-1})_{ij}$, $\Sigma = \sigma(x)\sigma^\top(x)$ denotes the noise covariance matrix. 
The action $S_T[\varphi]$ measures the
difficulty of the passage from state $x_0$ to the vicinity of state $x_1$ via path $\varphi$ in
the following probabilistic sense~{\cite{freidlin_random_2012}}:
\[ \lim_{\delta \downarrow 0} \lim_{\varepsilon \downarrow 0} - \varepsilon
{\log P} (|X^{\varepsilon} - \varphi | < \delta) \approx
\inf_{T > 0} S_T [\varphi] . \]

For rare event problems such as transitions between metastable states, the quantity of primary interest is the minimal cost required to move from state $x_0$ to state $x_1$, which leads to the definition of the \emph{quasi-potential}~{\cite{freidlin_random_2012}}:
\begin{equation} \label{eq:Quai-Potential}
V(x_0,x_1)
= \inf_{T>0}
\inf_{\varphi \in C_{x_0,x_1}(0,T)}
S_T[\varphi].
\end{equation}
We use $T^{\ast}$ and $\varphi^{\ast}$ denote the minimizers of $S_T[\varphi]$. The minimizing trajectory $\varphi^*$ is referred to as the MAP. The core of quasi-potential calculation lies in solving a variational minimization problem to obtain the MAP, a process collectively known as the \textit{minimum action method}~\cite{e_minimum_2004}. However, a significant numerical challenge arises when the MAP passes through or originates from a stationary point of the deterministic system: the optimal transition time $T^*$ theoretically tends toward infinity. For instance, in escape problems, the system must slowly "climb" against the deterministic drift field $b(x)$ starting from a stable equilibrium point $x_0$. Since the deterministic velocity is zero at the equilibrium point, the evolution of the path during its initial departure is extremely slow, which mathematically manifests as the residence time of the trajectory near $x_0$ becoming infinite.

To accurately resolve this long-term asymptotic behavior numerically, it is common to employ time reversal and domain mapping. Specifically, considering an escape path from a stable equilibrium $x_0$ to a non-stationary point $x_1$, the original problem is formulated as:
\[
    \min_{\substack{\varphi(\tau = -\infty) = x_1 \\ \varphi(\tau = 0) = x_0}} \frac{1}{2} \int_{-\infty}^{0} \|\varphi_{\tau} - \tilde{b}(\varphi)\|_{\Sigma^{-1}}^2 \diff\tau.
\]
By introducing the time reversal $t = -\tau$ and defining the reversed drift field $b(\varphi) = -\tilde{b}(\varphi)$, the problem originally defined on $(-\infty, 0]$ is mapped onto the semi-infinite interval $[0, \infty)$:
\[
    \min_{\substack{\varphi(t = \infty) = x_1 \\ \varphi(t = 0) = x_0}} \frac{1}{2} \int_{0}^{\infty} \|\varphi_t - b(\varphi)\|_{\Sigma^{-1}}^2 \diff t.
\]
Under this formulation, the calculation of the MAP is transformed into a variational problem on $[0, \infty)$. Traditional local discretization methods, such as finite differences, often require artificial truncation of the infinite domain or complex adaptive meshing to handle the asymptotic boundary conditions as $t \to \infty$. In contrast, the Laguerre spectral method leverages the natural decay properties of its basis functions at infinity, enabling the precise capture of the exponential dynamical characteristics near stationary points without the need for domain truncation.

\section{The Euler--Lagrange approach and numerical strategies}
\label{Numerical Strategies}
Based on the problem reformulation on the semi-infinite domain introduced in Section \ref{Problem Formulation}, the action functional for the transformed system is given by:
\begin{equation}
    \label{eq:MAMRp} S[\varphi] = \frac{1}{2} \int_0^{\infty} \|\varphi_t - b(\varphi)\|_{\Sigma^{-1}}^2 \diff t.
\end{equation}
For simplicity, we assume $\Sigma$ to be the identity matrix and decompose the drift field as 
\[b(\varphi) = L\varphi + f(\varphi),\]
where $L \varphi$ is the major linear part of $b (\varphi)$. The first variation of $S[\varphi]$ with respect to $\varphi$ in the direction of any test function $\psi \in H_0^1(\mathbb{R}^+)$ leads to the optimality condition:
\begin{equation}
    \label{var} \Bigl( \frac{\delta S}{\delta \varphi}, \psi \Bigr) = \int_0^{\infty} \langle \varphi_t - L\varphi - f(\varphi), \psi_t - L\psi - f'(\varphi)\psi \rangle \,\mathrm{d}t = 0,
\end{equation}
where the space $H_0^1(\mathbb{R}^+)$ is defined as \[H_0^1(\mathbb{R}^+) = \left\{ \varphi \in H^1(\mathbb{R}^+) \mid \varphi(0) = \varphi(\infty) = 0 \right\}.\]
Integration by parts yields the corresponding second-order Euler--Lagrange (E--L) equation (\cite{evans2022partial}):
\begin{equation}
    -\varphi_{tt} + L\varphi_t + f'(\varphi)\varphi_t - (L + f'(\varphi))^T \left(\varphi_t - L\varphi - f(\varphi)\right) = 0. \label{eq:EL1}
\end{equation}
By partitioning the terms into linear and nonlinear operators, the E--L equation can be simplified as 
\[
\mathcal{G}(\varphi) = {A}(\varphi) + \mathcal{N}(\varphi) = 0, 
\]
where
\begin{align}
    {A}(\varphi) &= -\varphi_{tt} + (L - L^T)\varphi_t + L^T L \varphi, \label{eq:linear_op} \\
    \mathcal{N}(\varphi) &= (f'(\varphi) - f'(\varphi)^T)\varphi_t + f'(\varphi)^T (L\varphi + f(\varphi)) + L^T f(\varphi), \label{eq:nonlinear_op}
\end{align}
where $f'(\varphi)$ denotes the Jacobian matrix of the nonlinear part $f$.

The above formulation focuses on ODE systems. For spatially extended systems governed by PDEs, we first performs a spatial semi-discretization  (e.g., finite differences/elements or spectral methods) to obtain a high-dimensional ODE system. The minimum action problem is then formulated for the resulting ODE dynamics, and the transition path is computed in the reduced setting. To solve this variational problem, two numerical paradigms are generally considered:
\begin{enumerate}
    \item \textbf{Discretization-then-variation}: This approach seeks a finite-dimensional approximation $\varphi_h \in V_N(\mathbb{R}^+)$ by directly minimizing the discretized action functional $S[\varphi_h]$ \eqref{eq:MAMRp}. This method is widely utilized in traditional MAM implementations (\cite{e_minimum_2004,wan_dynamic_2017}).
    \item \textbf{Variation-then-discretization}: This approach first derives the E--L equation \eqref{eq:EL1} and subsequently solves the resulting boundary value problem using iterative schemes (see e.g. \cite{heymann_geometric_2008}).
\end{enumerate}
Under appropriate anti-aliasing conditions, these two paradigms can be shown to be equivalent.

We solve the nonlinear system $\mathcal{G}(\varphi)=0$ by a stabilized linearly semi-implicit iteration (see, e.g. \cite{saad2003iterative,shen_numerical_2010, WangYu2018EfficientSecond}),
\begin{equation}\label{eq:semiimplicit}
\frac{\varphi^{n+1}-\varphi^{n}}{\tau}+ {A} \varphi^{n+1}=- \mathcal{N} (\varphi^{n}),
\end{equation}
which can be interpreted as a preconditioned gradient-type method with $(I+\tau A)^{-1}$ acting as the preconditioner. Since $A$ is linear and independent of $\varphi$, the matrix $I+\tau A$ can be assembled once and reused throughout the iteration. Under the Laguerre spectral discretization, $I+\tau A$ is sparse (see next section for details) and can therefore be inverted efficiently.

Moreover, when the linear operator $L$ admits a diagonal representation in the chosen Laguerre basis, we further obtain a fully diagonalized Laguerre spectral scheme, in which the update reduces to element-wise operations. This significantly lowers the per-iteration cost and is particularly advantageous for large-scale dynamics. 

\section{Laguerre spectral method}
\label{Laguerre Spectral Method}

In this section, we develop the Laguerre spectral discretization for the Euler--Lagrange equations derived previously. The choice of Laguerre-type basis functions is motivated by their natural orthogonality and decay properties on the semi-infinite interval $[0, \infty)$.

\subsection{Laguerre polynomials and Laguerre functions}
\label{Laguerre Polynomials and Laguerre Functions}
The standard Laguerre polynomials, denoted by $\mathscr{L}_k(t)$, are defined as the eigenfunctions of the following Sturm--Liouville problem (see e.g.
{\cite{shen_spectral_2011}}, Chapter 7):
\[
    t y'' + (1 - t) y' + ky = 0, \quad t \in [0, \infty).
\]
These polynomials can be generated via the following three-term recurrence relation:
\begin{equation}\label{eq:std-3term}
\begin{aligned}
    &\LaP_0(t) = 1, \quad
    \LaP_1(t) = 1 - t, \\
    &(k + 1) \LaP_{k+1}(t)  = (2k + 1 - t) \LaP_k(t) - k \LaP_{k-1}(t), \quad k \ge 1.    
\end{aligned}
\end{equation}
The Laguerre polynomials are orthogonal with respect to the weight function $w(t) = e^{-t}$ on the interval $[0, \infty)$, satisfying:
\[
    \int_0^{\infty} \LaP_k(t) \LaP_j(t) e^{-t} \diff t = \delta_{kj},
\]
where $\delta_{kj}$ denote the Kronecker delta function. The Laguerre polynomials have following boundary value
\begin{equation}\label{eq:LaP-bv}
    \LaP_k(0) = 1, \quad \text{ for all } k \ge 0,
\end{equation} 
which will be used to build basis with zero boundary value. 
Furthermore, they satisfy several differentiation identities:
\begin{align}\label{eq:LaP-derivatives}
    \begin{aligned}
    & \LaP_n'(t) = -\sum_{k=0}^{n-1} \LaP_k(t), \\
    & \LaP_n'(t) - \LaP_{n+1}'(t) = \LaP_n(t), \quad
    t {\LaP}_n'(t) = n(\LaP_n(t) - \LaP_{n-1}(t)).
    \end{aligned}
\end{align}
To handle the variational problem in $L^2(0, \infty)$ more effectively, it is suggested by Shen~\cite{Shen2000StableEfficient} that one should utilize the Laguerre functions $\LaF_k(t)$, which are defined by pre-multiplying the Laguerre polynomials with an exponential decay factor:
\begin{equation}\label{eq:LaF}
    \LaF_k(t) := \LaP_k(t) e^{-t/2}.
\end{equation}
A key advantage of this transformation is that the weight function is absorbed into the basis, resulting in orthogonality with respect to the standard $L^2$ inner product:
\[
    \int_0^{\infty} \LaF_k(t) \LaF_j(t) \,\mathrm{d}t = \delta_{kj},
\]
while maintaining $\LaF_k(0) = 1$. From \eqref{eq:LaP-derivatives}, we obtain the following differential properties of $\LaF_k(t)$:
\begin{equation}\label{eq:LaF_diff}
\begin{aligned}
    & \LaF_n'(t) - \LaF_{n+1}'(t) = \frac{1}{2} (\LaF_n(t) + \LaF_{n+1}(t)), \\
    & t \LaF_n'(t) = -\frac{n}{2} \LaF_{n-1}(t) - \frac{1}{2} \LaF_n(t) + \frac{n+1}{2} \LaF_{n+1}(t).
\end{aligned}
\end{equation}
These identities will be used to construct Laguerre bases with sparse differentiation matrices.

\subsection{Laguerre spectral Galerkin method for linear operators}

To solve the Euler--Lagrange equations, we employ a spectral Galerkin scheme. For a function $x(t) \in H^1([0, \infty))$ with the initial condition $x(0) = x_0$, we consider the following expansion:
\begin{equation}
    x(t) = \sum_{k=0}^{N} \hat{x}_k \phi_k(t),
\end{equation}
where the basis functions $\{\phi_k(t)\}$ are constructed from the Laguerre functions $\LaF_k(t)$ to satisfy the underlying functional space requirements~\cite{ChenShenYu2012NewSpectral}:
\[
    \phi_0(t) = \LaF_0(t), \quad \phi_k(t) = \LaF_k(t) - \LaF_{k-1}(t), \quad k = 1, 2, \ldots, N.
\]
By utilizing the properties of Laguerre functions, the corresponding mass matrix $M$, stiffness matrix $S$, and first-order differentiation matrix $D$ can be derived analytically:
\begin{align*}
m_{kj} &= \int_0^{\infty} \phi_k(t)\phi_j(t)\diff t
= \left\{\;
\begin{aligned}
&1,  && k=j=0,\\
&2,  && k=j\neq 0,\\
&-1, && |k-j|=1,\\
&0,  && \text{otherwise},
\end{aligned}
\right. \\[1mm]
s_{kj} &= \int_0^{\infty} \phi_k'(t)\phi_j'(t)\diff t
= \left\{\;
\begin{aligned}
&1/4, && k=j=0,\\
&1/2, && k=j\neq 0,\\
&1/4, && |k-j|=1,\ k\neq0,\ j\neq0,\\
&0,        && \text{otherwise},
\end{aligned}
\right. \\[1mm]
d_{kj} &= \int_0^{\infty} \phi_k'(t)\phi_j(t)\diff t
= \left\{\;
\begin{array}{ll}
-1/2, & k=j=0,\\
0,    & k=j\neq 0,\\
1/2,  & j=k+1,\\
-1/2, & j=k-1,\\
0,    & \text{otherwise}.     
\end{array}
\right.
\end{align*}
In the MAM framework, since the initial state $x_0$ is fixed, we adopt a lifting technique similar to the treatment of non-homogeneous Dirichlet boundary conditions in finite element methods. Specifically, we decompose the path as $x(t) = x_0 \phi_0(t) + \varphi_0(t)$, where $\varphi_0(t)$ vanishes at $t=0$ and is expanded using the remaining basis functions $\{\phi_k\}_{k=1}^N$. This approach allows us to solve for the second part while automatically satisfying the boundary condition $x(0)=x_0$.

Now, we describe the Laguerre method for the linear operator $A (\varphi)$. 
It is known that that resolution and efficiency of Laguerre spectral methods are significantly influenced by the spatial decay rate of the basis functions. To account for this, we introduce a linear time scaling $t = \beta \tau$, where $\tau \in [0, \infty)$ is the computational time and $\beta > 0$ is the scaling parameter. Under this transformation, the derivative and the action functional are reformulated as:
\begin{equation}\label{eq:scaledAction}
    \partial_t = \frac{1}{\beta} \partial_\tau, \quad S[\varphi(\tau), \beta] = \frac{1}{2} \int_0^{\infty} \bigl\| \frac{1}{\beta} \frac{d\varphi}{d\tau} - b(\varphi) \bigr\|^2 \beta d\tau.
\end{equation}
This formulation demonstrates an analogous linear scaling to the tMAM method~\cite{wan_minimum_2015}, yet it maintains finite $\beta$ values on the original domain $[0, \infty)$. 

The linear operator taking into account the scaling parameter $\beta$ is:
\[
	A_{\beta} (\varphi) = - \dfrac{1}{\beta} \varphi_{tt} + (L - L^T) \varphi_t
	+ \beta L^T L \varphi, \quad A_{\beta} (\varphi) = A_{\beta} (\varphi_0) +
	A_{\beta}  (x_0 \phi_0 (t)),
\]
where the last term $A_{\beta}  (x_0 \phi_0 (t))$ is a fixed term. Applying the Galerkin condition, the weak form is given by:
\[
    (A_{\beta}(\varphi), \psi) = \frac{1}{\beta} (\varphi_\tau, \psi_\tau) + (L - L^T)(\varphi_\tau, \psi) + \beta (L^T L \varphi, \psi) = 0, \quad \forall \psi \in V_N^0.
\]
Here $V_N^0=\text{span}\{ \phi_k, k=1,\ldots, N \}$.
Let $X \in \mathbb{R}^{m \times N}$ be the matrix of coefficients where each row corresponds to a component of the $m$-dimensional path $\varphi_0$. Using the previously defined matrices $M, S$, and $D$, the discrete weak form is:
\[
    (A_{\beta}(\varphi), \psi) = \frac{1}{\beta} XS + (L - L^T) XD + \beta L^T L XM + F_{x_0}.
\]
Here $\psi$ stands for a column vector composed of all basis functions.
$F_{x_0}$ is related to boundary term. To facilitate the numerical solution, we vectorize $X$ into $\mathbf{x} = \text{vec}(X) \in \mathbb{R}^{dN \times 1}$. Utilizing the properties of the Kronecker product, the linear system for $\varphi_0$ is expressed as:
\begin{equation}
		\label{linsys} (A_{\beta} (\varphi_0), \psi) = \dfrac{1}{\beta} S \otimes Ix
		+ D \otimes (L^T - L) x + \beta M \otimes (L^T L) x,
	\end{equation}
where $I$ is the identity matrix of dimension $m$. The boundary term $F_{x_0}$ arises from the fixed component $x_0 \phi_0(t)$ and is computed as:
\begin{equation}
    \label{bound} F_{x_0} = \bigl[ \frac{1}{\beta} s_{01} x_0 + d_{01} (L^T - L) x_0 + \beta m_{01} (L^T L) x_0 \bigr] \otimes e_1,
\end{equation}
where $e_1$ is the first unit vector in the coefficient space.

\subsection{Fully diagonalized Laguerre spectral method}
Although the Laguerre spectral method has high accuracy, due to the fact that the mass matrix $M$ and the stiffness matrix $S$ are tri-diagonal rather than diagonal, directly solving \eqref{linsys} involves complex matrix operations (especially Kronecker product), which becomes slow when dealing with high-dimensional systems (such as partial differential equations). To alleviate this, we propose a matrix decomposition method to fully diagonalize the linear operator, transforming the coupled system into a series of independent scalar equations \citep{SHEN2007710}.

Consider the generalized eigenvalue problem associated with the stiffness and mass matrices:
\begin{equation}
\label{eigen} Sx = \lambda Mx.
\end{equation}
Since both $S$ and $M$ are symmetric positive definite within the Laguerre framework, all the eigenvalues of (\ref{eigen}) are real and positive. Let $\Lambda = \text{diag}(\lambda_1, \dots, \lambda_N)$ be the diagonal matrix of eigenvalues and $P$  be the matrix whose columns are the
corresponding eigenvectors of (\ref{eigen}). Then (\ref{eigen}) becomes
\begin{equation}
	\label{decomp} SP = MP \Lambda .
\end{equation}
One can explain (\ref{decomp}) in another way. The matrix $P$ has the
properties that can diagonalize the mass and stiffness matrix simultaneously:
\[
P^T MP = I, \quad P^T S P = \Lambda.
\]
That is to say, let $p_{ij}$ be the entries of $P$, we can use the following
approximation
\[ x (t) = \sum_{k = 0}^{\infty} \bar{x}_k \psi_k (t), \]
alternatively, where $\psi_j (t) = \sum_i p_{ij} \phi_i (t)$ and $\hat{x}_i =
\sum_j p_{ij}  \bar{x}_i$.
In this case, thanks to the diagonal of $L$, (\ref{linsys}) becomes:
\[
	(A_{\beta} (\varphi_0), \psi) = \dfrac{1}{\beta} \Lambda \otimes I \hat{x} +
	\beta I \otimes (L^T L)  \hat{x},
\]
where $\hat{x}$ denotes the vector form of the discretized $\varphi_0$ with
each column corresponding to the coefficients $\bar{x}_i$. Then the linear
part of the MAM system becomes diagonal and simply to solve. 

\subsection{Numerical quadrature for nonlinear terms}

While the linear operators are discretized analytically in the spectral domain, the nonlinear drift term $b(\varphi)$ typically requires evaluation in the physical space.  To ensure the exact integration of nonlinear terms, particularly for the quadratic nonlinearities inherent in the Burgers or Navier-Stokes equations, one must account for the increased degree of the integrand. Specifically, if the path $\varphi$ is represented by a polynomial of degree $N$, its quadratic products result in terms of degree up to $4N$ in the action functional. Consequently, the standard Gauss quadrature with $N$ points is insufficient due to aliasing errors; instead, a Gauss-Laguerre-Radau quadrature with $M \ge 2N+1$ points is required to maintain spectral precision. The computation of the corresponding quadrature points and weights involves determining the zeros of $L_{2N+1}(x)$, a procedure detailed in the standard spectral literature (\cite{shen_spectral_2011}, Chapter 7.1.3). However, the standard procedure has stability issue for large $N$. In our implementation, these nodes and weights are efficiently computed using improved Laguerre spectral algorithms developed in~\cite{huang2024improved}.

Since in the numerical implementation,  we use Laguerre functions as basis functions, so $\varphi$ and $b (\varphi)$ are not polynomials, but polynomials multiplied by some exponential function. This necessitates a specialized treatment of the quadrature weights to account for the varying decay rates in the integrand. Specifically, for terms involving powers of $\varphi$, such as $\int_0^{\infty} \varphi^{\alpha}(t) \diff t$, the integral is reformulated as:
\[
    \int_0^{\infty} \varphi^{\alpha}(t) \diff t = \int_0^{\infty} \left[ \psi(t) \right]^{\alpha} e^{-\alpha t / 2} \diff t, \quad \psi(t) \in P_N.
\]
For the linear case ($\alpha=2$), the term corresponds to the standard $L^2$ norm and can be computed analytically or via standard Gauss--Laguerre quadrature. However, for higher-order nonlinearities ($\alpha > 2$), we employ a generalized Laguerre-Gauss-Radau quadrature:
\[
    \int_0^{\infty} \varphi^{\alpha}(t) \diff t \approx \sum_{j=0}^{M} b(\varphi(\xi_j^{(\alpha)})) \hat{\omega}_j^{(\alpha)},
\]
where $\{\xi_j^{(\alpha)}\}_{j=0}^{2N}$ and $\{\hat{\omega}_j^{(\alpha)}\}_{j=0}^{2N}$ denote the quadrature points and weights adjusted to the exponential weight function $e^{-\alpha t/2}$. 

As detailed in \cite{huang2024improved}, the above pseudo-spectral procedure for nonlinear terms helps maintain stability in double-precision calculations when thousands of Laguerre basis functions are used, and provides the accuracy required for the nonlinear MAM iterations.

\subsection{The choice of scaling factor}


The choice of scaling factor $\beta$ (see \eqref{eq:scaledAction}) is critical to the efficiency of the Laguerre method. Different scaling options are proposed. Tang~\cite{tang_hermite_1993} proposed a scaling approach for Hermite method by making all Gauss-Hermite quadrature points inside the effective spatial region of functions to be approximated, which also works for Laguerre method. Xia et al.~\cite{XiaSC2021EfficientScaling} introduced a frequency-dependent scaling strategy. Huang and Yu~\cite{huang2024improved} gave an explicit formula for the optimal scaling of Laguerre approximation of exponential decay functions with oscillation. Hu and Yu~\cite{hu2026scaling,huscaling} recently established a comprehensive error estimate framework for scaled Hermite and Laguerre approximation, providing a more intrinsic way to choose optimal scaling factor.
In this paper, building upon the error analysis framework for Laguerre methods \citep{huscaling}, we implement a dynamic adaptation mechanism to optimize the parameter $\beta$ throughout the iterative process. The core of this strategy lies in equilibrating the spatial and spectral resolution of the Laguerre basis. 

The method estimates the truncation error by monitoring the solution's behavior in the far-field or "tail" region $[x_K, \infty)$ of the domain, where $x_K$ is chosen based on the distribution of the Laguerre-Gauss-Radau quadrature points. The adjustment logic is governed by the relative magnitude of the spatial decay error versus the spectral interpolation error:
\begin{itemize}
    \item When the transition path exhibits broad spatial features or significant activity in the far-field, the truncation error in the tail region becomes dominant. In such cases, $\beta$ is increased to "stretch" the basis functions, thereby extending their effective spatial region and improving spatial coverage.
    \item Conversely, if the system develops high-frequency oscillations or steep gradients nor far away from the origin that are unresolved by the current basis, $\beta$ is reduced. This effectively "compresses" the basis functions toward the origin, concentrating the spectral resolution to better resolve local transient features.
\end{itemize}
By adapting the scaling factor during the optimization process, the LMAM remains robust and accurate even when the characteristic scales of the transition path change significantly. The complete adaptive update strategy is summarized in Algorithm~\ref{alg:beta_update}.

\begin{algorithm}[htbp]
\caption{Adaptive $\beta$ Update} \label{alg:beta_update}
\begin{algorithmic}[1]
\Require Current solution $\varphi$, baseline $b(\varphi)$, threshold $\tau=0.3$, previous $\beta$
\Ensure Updated $\beta$ value, projected solution $\varphi$

\State Compute total residual norm: $E_\text{tot} \gets \|\varphi_t - b(\varphi)\|_{L^2( [0,\infty])}$ 
\State Compute far-field residual norm: $E_x \gets \|\varphi_t - b(\varphi)\|_{L^2( [x_k,\infty])}, \ k\sim\lfloor N/3 \rfloor$
\State Compute residual corresponding to frequency truncation:  $E_f \gets E_\text{tot} - E_x$
\State Calculate scaling ratio: $r \gets \frac{E_f}{E_x}$ \Comment{Frequency-to-space ratio}
 
\If{$|r - 1| > \tau$} \Comment{Threshold check}
    \State $\delta \gets \text{clip}(\log r, -\tau, \tau)$
    \State $\beta_{\text{new}} \gets \beta / \exp(0.5\delta)$
    \State $\beta \gets 0.8\beta + 0.2\beta_{\text{new}}$ \Comment{Exponential moving average}
\EndIf
\State Project the solution $\varphi$ to the solution space using new $\beta$.
\State \Return $\beta$ and project solution $\varphi$
\end{algorithmic}
\end{algorithm}

The overall LMAM algorithm is summarized in Algorithm \ref{alg:LMAM}.

\begin{algorithm}[htbp]
\caption{LMAM} \label{alg:LMAM}
\begin{algorithmic}[1]
\Require Initial $\beta$, maximum iterations \texttt{maxiter}, tolerance \texttt{tol},update tolerance $\epsilon$
\Ensure Optimized solution $\varphi_{\text{new}}$

\State Initialize $\varphi$
\State Precompute matrix $A$ for linear part of $G(\varphi)$ 
\For{$\text{it} = 1$ \textbf{to} \texttt{maxiter}}
    
    \State Compute linear gradient $A\varphi$
    \State Compute nonlinear part $\mathcal{N}(\varphi)$ 
    \State $G(\varphi) \gets A\varphi +\mathcal{N}(\varphi)$ \Comment{Combine linear and nonlinear parts}
    \State Compute $\varphi_{\text{new}}$ via optimization method (Solve \eqref{eq:semiimplicit} or use it as a preconditioner)
    \State $\text{resv} \gets \|\varphi_{\text{new}} - \varphi\|$ \Comment{Residual calculation}
	\If {$\text{resv} < \epsilon$}
		\State  Adapt $\beta$ using Algorithm \ref{alg:beta_update} \Comment{Adaptive $\beta$ selection}
	\EndIf
    \If{$\text{resv} < \texttt{tol}$}
        \State \Return $\varphi_{\text{new}}$
    \Else
        \State $\varphi \gets \varphi_{\text{new}}$ \Comment{Update solution}
    \EndIf
\EndFor
\State \Return $\varphi_{\text{new}}$ \Comment{Fallback return after maxiter}
\end{algorithmic}
\end{algorithm}

\section{Convergence analysis of LMAM}
\label{Spectral Properties}
\subsection{Convergence  for linear systems}
\label{Convergence  for Linear Systems}
We first present a sharp convergence analysis of LMAM for linear systems. This setting admits an explicit characterization of the minimum
action path and allows us to rigorously quantify the spectral convergence behavior of the Laguerre
projection, thereby revealing the fundamental mechanism underlying the efficiency of the proposed
method.
\subsubsection{Exact solution of linear minimum action problem }

We consider the linear minimum action problem
\begin{equation}
S[\phi]
=
\frac{1}{2}
\int_0^{+\infty}
\|\dot{\phi}(t) - B \phi(t)\|^2 \diff t,
\label{eq:linear_action}
\end{equation}
where $B \in \mathbb{R}^{m \times m}$ is a constant matrix and $\phi(0)=x_0$.
Following classical results (see, e.g.,{\cite{chen_asymptotic_2006}}), the unique minimizer of \eqref{eq:linear_action}
can be expressed explicitly.

\begin{lemma}
\label{lem:linear_exact}
Assume that all eigenvalues of the matrix
\[
Q := \frac12 \Bigl( \int_0^{+\infty} e^{Bt} e^{B^{T}t} \diff t \Bigr)^{-1}
\]
have positive real parts. Then the MAP is given by
\[
\phi(t) = e^{-(2Q+B)t} x_0, \qquad t \in [0,+\infty).
\]
\end{lemma}

This explicit form provides a natural benchmark for assessing the approximation properties of
Laguerre spectral discretizations.

\subsubsection{Laguerre projection and exponential convergence }

Let $\{\widetilde \LaP_n\}_{n\ge0}$ denote the Laguerre function basis defined in Section \ref{Laguerre Polynomials and Laguerre Functions}, and introduce
the scaled variable $\tau = \lambda t$ with $\lambda = \beta^{-1}$,
here we denote $\widehat{Q} = - 2 Q - B.$ For the Laguerre minimum action
method, set the projection approximation representation
\[ \varphi_h (t) = \sum_{n = 0}^N \widehat{q}_n \LaP_n (\lambda t) {e^{-
		\frac{\lambda t}{2}}} , \quad \lambda = \frac{1}{\beta} . \]
Here
\[  \widehat{q}_n = \lambda \int^{\infty}_0 \varphi (t) \LaP_n (\lambda t) {e^{-
		\frac{\lambda t}{2}}}  \mathd t. \]
Then we have the following result.
\begin{lemma}
	For $\varphi(t)=e^{-\widehat{Q}t}x_0$, assume that all eigenvalues of $\widehat{Q}$ have positive real parts. Then in the projection approximation, we have
	\[  \widehat{q }_n = \bigl( \frac{\hat{Q}}{\lambda} + \frac{1}{2} I
	\bigr)^{- 1} \bigl( I - \bigl( \frac{\hat{Q}}{\lambda} + \frac{1}{2} I
	\bigr)^{- 1} \bigr)^n x_0, \]
	where $I$ represents the identity matrix.
\end{lemma}

\begin{proof}
	Because of
	\[ \widehat{q }_n = \lambda \int^{\infty}_0 \varphi (t) \LaP_n (\lambda t)
	{e^{- \frac{\lambda t}{2}}}  \mathd t. \]
	Substitute $\varphi(t)=e^{-\widehat{Q}t}x_0$ into the above expression and $\lambda t$ into
	$\tau$, then we obtain
	\[ \widehat{q }_n = \int^{\infty}_0 \LaP_n (\tau) {e^{- \bigl(
			\frac{\hat{Q}}{\lambda} + \frac{1}{2} I \bigr) \tau}}  x_0 \mathd \tau .
	\]
	For the sake of simplicity, we denote $R = \frac{\hat{Q}}{\lambda} +
	\frac{1}{2} I$, and
	\[ \widehat{Q }_n = \int^{\infty}_0 \LaP_n (\tau) {e^{- R \tau}}  \mathd \tau,
	\]
	By the assumption, $R$ is invertible. The only thing we need to do is to
	calculate $\widehat{Q }_n$, and $\widehat{q }_n = \widehat{Q }_n x_0$. It is
	easy to get
	\[ \frac{\mathd^k}{\mathd \tau^k } \LaP_n (0) = (- 1)^k \binom{n}{n-k}, \quad k \in [0, n] \cap \mathbb{N} \]
	and
	\[ \frac{\mathd^{n + 1}}{\mathd \tau^{n + 1} } \LaP_n (0) = 0. \]
	By the integration by parts, we have
	\[ \int^{\infty}_0 \LaP_n (\tau) {e^{- R \tau}}  \mathd \tau = \sum_{k = 0}^n -
	R^{- (k + 1)} e^{- R t}  \frac{\mathd^k}{\mathd
		\tau^k } \LaP_n (\tau) \Big|_{\tau = 0}^{\infty}
	\]
	By the binomial theorem, we obtain
	\[ \widehat{Q}_k = R^{- 1} (I - R^{- 1})^k. \qedhere \] 
\end{proof}

In particular, the LMAM exhibits \emph{geometric (exponential) convergence} with respect to
the number of Laguerre modes $N$.

\subsubsection{Optimal scaling and convergence factor}

As a direct consequence, the truncation error of the Laguerre approximation can be estimated
explicitly.

\begin{theorem}
	For the linear minimum action problem, the projection approximation error
	can be represented by
    \begin{equation}\label{eq:linear_error}
       \| \varphi (t) - \varphi_h  (t) \|_2 \leqslant \bigl\| \bigl( \hat{Q} +
	\frac{\lambda}{2} I \bigr)^{- 1} \bigr\|_2  \| x_0 \|_2 
    \frac{\bigl\| I - \bigl( \frac{\hat{Q}}{\lambda} + \frac{1}{2} I \bigr)^{-1}
		 \bigr\|^{N + 1}_2}
         {\sqrt{1 - \bigl\| I - \bigl( \frac{\hat{Q}}{\lambda}
			+ \frac{1}{2} I \bigr)^{- 1} \bigr\|^2_2}}. 
    \end{equation}
\end{theorem}

\begin{proof}
	Thanks to Parseval's identity,
	\begin{align*}
		\| \varphi (t) - \varphi_h  (t) \|^2_2 = & \lambda \sum_{n = N +
			1}^{\infty} \| \widehat{q }_n \|_2^2 & \\
		= & \lambda \sum_{n = N + 1}^{\infty} \bigl\| \bigl(
		\frac{\hat{Q}}{\lambda} + \frac{1}{2} I \bigr)^{- 1} \bigl( I - \bigl(
		\frac{\hat{Q}}{\lambda} + \frac{1}{2} I \bigr)^{- 1} \bigr)^n x_0
		 \bigr\|_2^2 & \\
		\leqslant & \lambda \sum_{n = N + 1}^{\infty} \bigl\| \bigl(
		\frac{\hat{Q}}{\lambda} + \frac{1}{2} I \bigr)^{- 1} \bigr\|^2_2 \bigl\|
		I - \bigl( \frac{\hat{Q}}{\lambda} + \frac{1}{2} I \bigr)^{- 1}
		 \bigr\|^{2 n}_2 \| x_0 \|^2_2 & \\
		= & \bigl\| \bigl( \hat{Q} + \frac{\lambda}{2} I \bigr)^{- 1}
		 \bigr\|^2_2  \frac{\bigl\| I - \bigl( \frac{\hat{Q}}{\lambda} +
			\frac{1}{2} I \bigr)^{- 1} \bigr\|^{2 (N + 1)}_2}{1 - \bigl\| I - \bigl(
			\frac{\hat{Q}}{\lambda} + \frac{1}{2} I \bigr)^{- 1} \bigr\|^2_2} \| x_0
		\|^2_2 &
	\end{align*}
	The conclusion is obtained by taking square root.
\end{proof}

Thus the order of projection error is related to $\left. \bigl\| I - \bigl(
	\frac{\hat{Q}}{\lambda} + \frac{1}{2} I \bigr)^{- 1} \right. \bigr\|_2 $, which we define as the theoretical 
	convergence factor for this problem.
	For
	\[ I - \bigl( \frac{\hat{Q}}{\lambda} + \frac{1}{2} I \bigr)^{- 1} = \bigl(
	\hat{Q} + \frac{\lambda}{2} I \bigr)^{- 1} \bigl( \hat{Q} -
	\frac{\lambda}{2} I \bigr), \]
	Use Schur decomposition of $\hat{Q} = U' T U$, where $U$ is a unitary matrix
	and $T$ is upper triangular, Then
    \begin{equation}\label{eq:linear_conv_factor_nongrad}
       \bigl\| I - \bigl( \frac{\hat{Q}}{\lambda} + \frac{1}{2} I \bigr)^{- 1}
	 \bigr\|_2 = \bigl\| \bigl( T + \frac{\lambda}{2} I \bigr)^{- 1} \bigl(
	T - \frac{\lambda}{2} I \bigr) \bigr\|_2 . 
    \end{equation}
	In the gradient case, $T$ is diagonal, denote $\lambda_1 $ and $\lambda_2 $
	are the maximum and minimum eigenvalue for $T$, then we obtain the optimal 
	$\lambda_{{\rm opt}} $ in this case is $2\sqrt{\lambda_1 \lambda_2 }$, and
    \begin{equation}
        \label{eq:linear_conv_factor_grad}
        \bigl\| I - \bigl( \frac{\hat{Q}}{\lambda_{{\rm opt}} } + \frac{1}{2} I
	\bigr)^{- 1} \bigr\|_2 = \frac{ \sqrt{\lambda_1} -
		\sqrt{\lambda_2}}{\sqrt{\lambda_1} + \sqrt{\lambda_2}}. 
    \end{equation}

This result provides a rigorous theoretical foundation for the adaptive scaling strategy employed in
LMAM, and explains the observed acceleration compared with fixed time scaling.

\subsection{Convergence properties for nonlinear systems}

We now discuss the convergence behavior of LMAM for general nonlinear systems.
Unlike the linear case, closed-form solutions are no longer available.
Nevertheless, meaningful convergence results can be obtained by combining local convexity
properties of the action functional with the approximation power of Laguerre spectral spaces.

\begin{assumption}[Local strong convexity / strong monotonicity]
\label{ass:local_strong_convex}
Let $\varphi_{\rm opt}$ be a (local) minimizer of the action functional $S[\varphi]$.
There exist constants $\alpha>0$ and $\delta>0$ such that for any
$\varphi_1,\varphi_2\in \mathcal{B}_\delta(\varphi_{\rm opt})
:=\{\varphi:\|\varphi-\varphi_{\rm opt}\|<\delta\}$,
the first variation $\partial S/\partial\varphi$ satisfies
\begin{equation}
\label{eq:strong_mono}
\Bigl(
\frac{\partial S}{\partial \varphi}(\varphi_1)
-\frac{\partial S}{\partial \varphi}(\varphi_2),
\ \varphi_1-\varphi_2
\Bigr)
\ge \alpha \|\varphi_1-\varphi_2\|^2 .
\end{equation}
\end{assumption}

\begin{assumption}[Local Lipschitz continuity of the first variation]
\label{ass:lipschitz_grad}
There exist constants $L>0$ and $\delta>0$ such that for any
$\varphi_1,\varphi_2\in \mathcal{B}_\delta(\varphi_{\rm opt})$,
\begin{equation}
\label{eq:lipschitz_grad}
\Bigl\|
\frac{\partial S}{\partial \varphi}(\varphi_1)
-\frac{\partial S}{\partial \varphi}(\varphi_2)
\Bigr\|
\le C_L \|\varphi_1-\varphi_2\| .
\end{equation}
\end{assumption}

\begin{theorem}
\label{thm:laguerre_apriori}
Define the weighted Sobolev space $\hat{B}_{a}^{m}(\mathbb{R}_{+})$ 
and corresponding weighted norm for the weight function ${\hat{\omega}_\alpha} = x^\alpha$ as
\[
\hat{B}_{a}^{m}(\mathbb{R}_{+}) := \left\{u \Big| \hat{\partial}_{x}^{k} u \in L_{\hat{\omega}_{a+k}}^{2}(\mathbb{R}_{+}), \, 0 \leq k \leq m\right\}, \quad \|u\|_{\hat{B}_{a}^{m}} := \Bigl( \sum_{k=0}^{m} \|\hat{\partial}_{x}^{k} u\|_{\hat{\omega}_{a+k}}^{2} \Bigr)^{1/2},
\]
where $\hat{\partial}_t \varphi := \bigl( \partial_t + \frac{1}{2} \bigr) \varphi$.
Suppose that $\hat{\partial}_t \varphi_\text{opt} \in \hat{B}_{0}^{m-1}(\mathbb{R}_+)$ and Assumption~\ref{ass:local_strong_convex} and~\ref{ass:lipschitz_grad} hold.
Let $\varphi_h\in V_N(\mathbb{R}^+)$ be the numerical solution
produced by the Laguerre MAM. Further assume that
$\varphi_h\in \mathcal{B}_\delta(\varphi_{\rm opt})$.
Then there exists a constant $C>0$, independent of $N$ (equivalently $h$),
such that
\begin{equation}
\label{eq:apriori_est}
\|\varphi_h-\varphi_{\rm opt}\|_1
\le C \sqrt{\frac{(N-m+1)!}{N!}} \|\hat{\partial}_t^{\,m}\varphi_{\rm opt}\|_{\hat{\omega}_{m-1}}.
\end{equation}
\end{theorem}

\begin{proof}
Let $\varphi_{\rm opt}\in H_0^1(\mathbb{R}_+)$ be a (local) minimizer of $S[\varphi]$ and
let $\varphi_h\in V_N(\mathbb{R}_+)$ be the LMAM approximation.
By the Euler--Lagrange equations for the continuous and discrete problems, we have
\[
\Bigl(\frac{\partial S}{\partial\varphi}(\varphi_{\rm opt}),\,\psi\Bigr)=0,
\quad \forall\,\psi\in H_0^1(\mathbb{R}_+),
\qquad
\Bigl(\frac{\partial S}{\partial\varphi}(\varphi_h),\,\psi_h\Bigr)=0,
\quad \forall\,\psi_h\in V_N^0.
\]

Under Assumption~\ref{ass:local_strong_convex} (strong monotonicity), with
$\varphi_1=\varphi_{\rm opt}$ and $\varphi_2=\varphi_h$, we obtain
\begin{equation}\label{eq:mono_step}
\alpha\|\varphi_{\rm opt}-\varphi_h\|^2
\le
\Bigl(\frac{\partial S}{\partial\varphi}(\varphi_{\rm opt})
-\frac{\partial S}{\partial\varphi}(\varphi_h),\,\varphi_{\rm opt}-\varphi_h\Bigr).
\end{equation}

Let $\Pi: H_0^1(\mathbb{R}_+)\to V_N^0(\mathbb{R}_+)$ be the projection defined by
\[
\bigl( (\varphi-\Pi\varphi)', \psi_h'\bigr) 
+ \frac14 ( \varphi-\Pi\varphi, \psi_h)=0,
\qquad \forall\,\psi_h\in V_N^0.
\]
Since $\Pi\varphi_{\rm opt}-\varphi_h\in V_N^0\subset H_0^1(\mathbb{R}_+)$,
the two Euler--Lagrange relations give
\[
\Bigl(\frac{\partial S}{\partial\varphi}(\varphi_{\rm opt}),\,\Pi\varphi_{\rm opt}-\varphi_h\Bigr)=0,
\qquad
\Bigl(\frac{\partial S}{\partial\varphi}(\varphi_h),\,\Pi\varphi_{\rm opt}-\varphi_h\Bigr)=0,
\]
and hence, by subtraction,
\begin{equation}\label{eq:orth_step}
\Bigl(\frac{\partial S}{\partial\varphi}(\varphi_{\rm opt})
-\frac{\partial S}{\partial\varphi}(\varphi_h),\,\Pi\varphi_{\rm opt}-\varphi_h\Bigr)=0.
\end{equation}
Combining \eqref{eq:orth_step} with $\varphi_{\rm opt}-\varphi_h=(\varphi_{\rm opt}-\Pi\varphi_{\rm opt})+(\Pi\varphi_{\rm opt}-\varphi_h)$ yields
\[
\Bigl(\frac{\partial S}{\partial\varphi}(\varphi_{\rm opt})
-\frac{\partial S}{\partial\varphi}(\varphi_h),\,\varphi_{\rm opt}-\varphi_h\Bigr)
=
\Bigl(\frac{\partial S}{\partial\varphi}(\varphi_{\rm opt})
-\frac{\partial S}{\partial\varphi}(\varphi_h),\,\varphi_{\rm opt}-\Pi\varphi_{\rm opt}\Bigr).
\]
Therefore, using Cauchy--Schwarz and Assumption~\ref{ass:lipschitz_grad} (Lipschitz continuity),
\begin{align*}
\Bigl|\Bigl(\frac{\partial S}{\partial\varphi}(\varphi_{\rm opt})
-\frac{\partial S}{\partial\varphi}(\varphi_h),\,\varphi_{\rm opt}-\Pi\varphi_{\rm opt}\Bigr)\Bigr|
& \le
\Bigl\|\frac{\partial S}{\partial\varphi}(\varphi_{\rm opt})
-\frac{\partial S}{\partial\varphi}(\varphi_h)\Bigr\|\,
\|\varphi_{\rm opt}-\Pi\varphi_{\rm opt}\| \\*
& \le
C_L\|\varphi_{\rm opt}-\varphi_h\|\,\|\varphi_{\rm opt}-\Pi\varphi_{\rm opt}\|.    
\end{align*}
Substituting into \eqref{eq:mono_step} and cancelling $\|\varphi_{\rm opt}-\varphi_h\|$ gives
\begin{equation}\label{eq:Galerkin_error_result}
\|\varphi_{\rm opt}-\varphi_h\|\le \frac{C_L}{\alpha}\,\|\varphi_{\rm opt}-\Pi\varphi_{\rm opt}\|.    
\end{equation}
Since $\hat{\partial}_t \varphi_\text{opt} \in \hat{B}_{0}^{m-1}(\mathbb{R}_+)$, then there exists a constant $C > 0$ independent of $N$, $m$, and $\varphi$ such that~(see \cite{shen_spectral_2011} Theorem 7.11)
	\begin{equation} \label{eq:proj-error}
	\| \varphi_\text{opt} - \Pi \varphi_\text{opt} \|_1 \leq C \sqrt{\frac{(N-m+1)!}{N!}} \| \hat{\partial}_t^m \varphi \|_{\hat{\omega}_{m-1}}.	    
	\end{equation}
The desired estimate is obtained by combining \eqref{eq:proj-error}	with \eqref{eq:Galerkin_error_result}.
\end{proof}


\section{Numerical experiments and applications}
\label{Numerical Experiments}




\subsection{Linear example}
This test is designed to verify that the Laguerre approximation exhibits the geometric convergence
predicted by the linear theory in Section~\ref{Convergence  for Linear Systems}, and to demonstrate
that the adaptive $\beta$ strategy automatically approaches a near-optimal time scaling.

We consider a two-dimensional linear gradient system
\begin{equation}\label{eq:exam-6.1}
    \left\{ \begin{array}{l}
		\dot{x}_1 = - 20 x_1 ,\\
		\dot{x_2} = - 0.2 x_2,
	\end{array} \right.
\end{equation}
    for which the optimal transition path is available analytically and the theoretical convergence factor
can be computed from the spectrum of the associated matrix $Q$ .

Table  \ref{table_linear} reports the agreement between the empirical and theoretical convergence factors for a range
of fixed $\beta$. The close match confirms that, for linear gradient systems, the LMAM
convergence is accurately characterized by the theory once rounding effects are excluded. We test several fixed values of $\beta$ and compare them with the adaptive $\beta$ strategy in Figure~\ref{linearfig}, where the path error decays geometrically as $N$ increases, and that the adaptive
strategy consistently achieves the best (or near-best) decay rate among all tested fixed scalings. This
demonstrates that the adaptive update effectively identifies an optimal temporal scaling factor ($\beta_\text{opt} = 0.25$) without
manual tuning.

\begin{table}[htbp]
    \centering
    \begin{tabular}{rcc}
            \toprule
            $\bm{\beta}$   & \textbf{numerical rate $r$} & \textbf{theoretical rate $r$}\\
            \midrule
            $ 0.10$ & 0.92281 & 0.92307\\
        
            $0.20$ & 0.85158& 0.85185\\
            
            $ 0.25$ & 0.81795 & 0.81818\\
           
            $ 0.50$ & 0.90475 & 0.90476\\
          
            $ 1.00$ & 0.95087 & 0.95121\\
            \bottomrule
    \end{tabular}
    \caption{The numerical and theoretical geometrical convergence rate $r^N$ for linear gradient case. By \eqref{eq:linear_error}, the theoretical 
    $r=\big\| I - ( {\hat{Q}}/{\lambda} + I/{2} )^{-1}\big\|_2$.}
    \label{table_linear}
\end{table}

\begin{figure}
\centering
\includegraphics[width=0.9\textwidth]{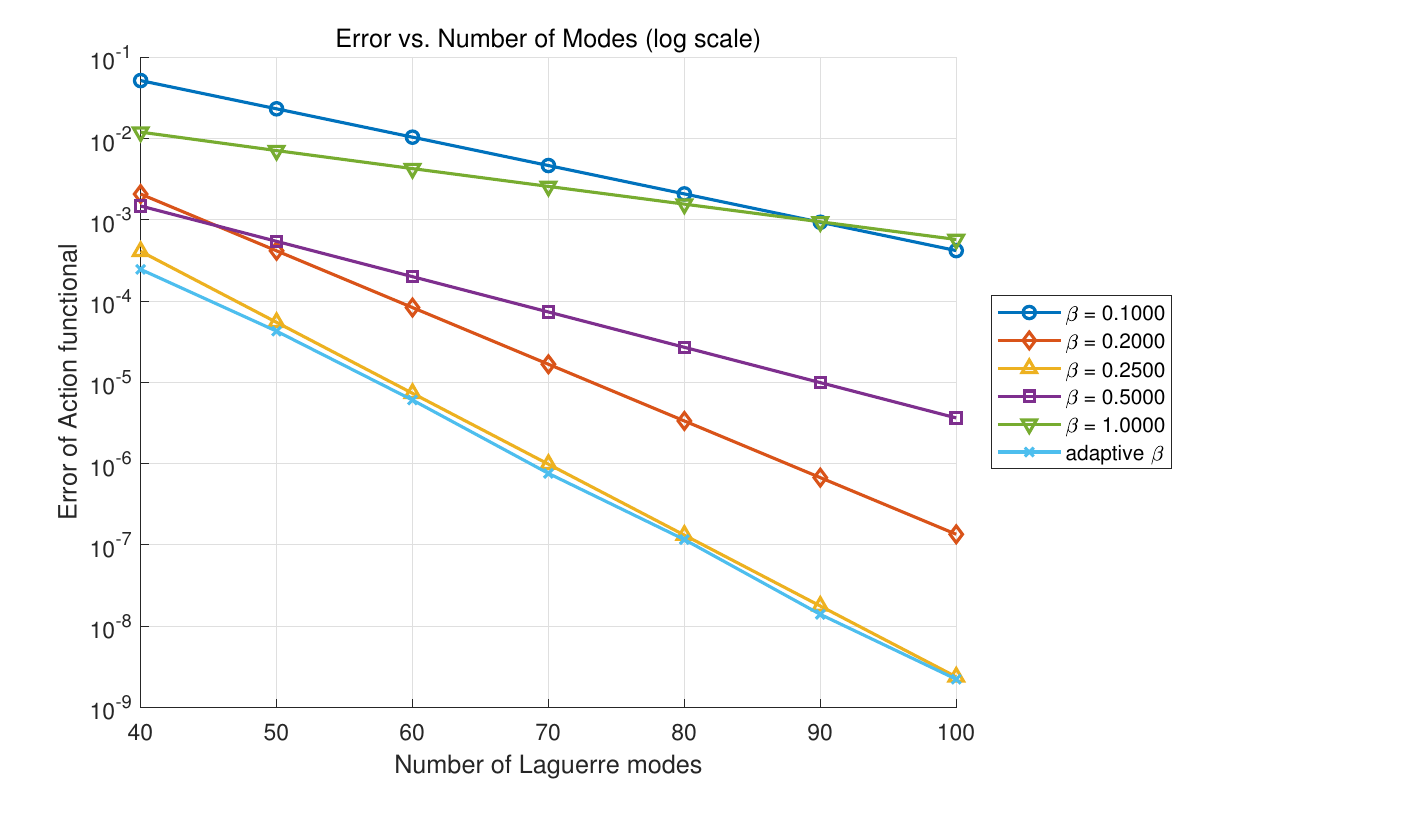}
\caption{The LMAM convergence of action functional for linear gradient example \eqref{eq:exam-6.1}. }
\label{linearfig} 
\end{figure}


\subsection{Nonlinear example}
Now we assess the performance of LMAM for a nonlinear gradient systems:
\begin{equation}\label{eq:ex6.2-nonlinear-grad-system}
    \begin{cases}
		\mathd x = -\partial_x V(x,y) \mathd t + \sqrt{\epsilon}\,\mathd W_t^x, \\
		\mathd y = -\partial_y V(x,y) \mathd t + \sqrt{\epsilon}\,\mathd W_t^y.
	\end{cases}
\end{equation}
where $W^X_t$ and $W^Y_t$ are independent Brownian motion processes, and the potential function $V(\cdot)$ is
\begin{equation}
	\label{app:example1}
	V(x,y) = \frac{1}{4}x^2 + \frac{1}{2}y^2 + xy^2 + x^2y.
\end{equation}
	The system admits a stable equilibrium at $a_1=(0,0)$.
We choose $a_2=(1,1)$ as the terminal point, which is not a critical point of the drift field.
This setting is particularly suitable for validation, since for gradient systems the quasi-potential
difference satisfies
\[
V(a_1,a_2) = 2\bigl(V(a_2)-V(a_1)\bigr),
\]
providing a reliable reference value for assessing the accuracy of the computed minimum action.
 The computational errors of the minimum action functional and their corresponding 
	 trajectory errors are shown in Figure \ref{fignonlinear}.

\begin{figure}
\centering
\includegraphics[width=0.9\textwidth]{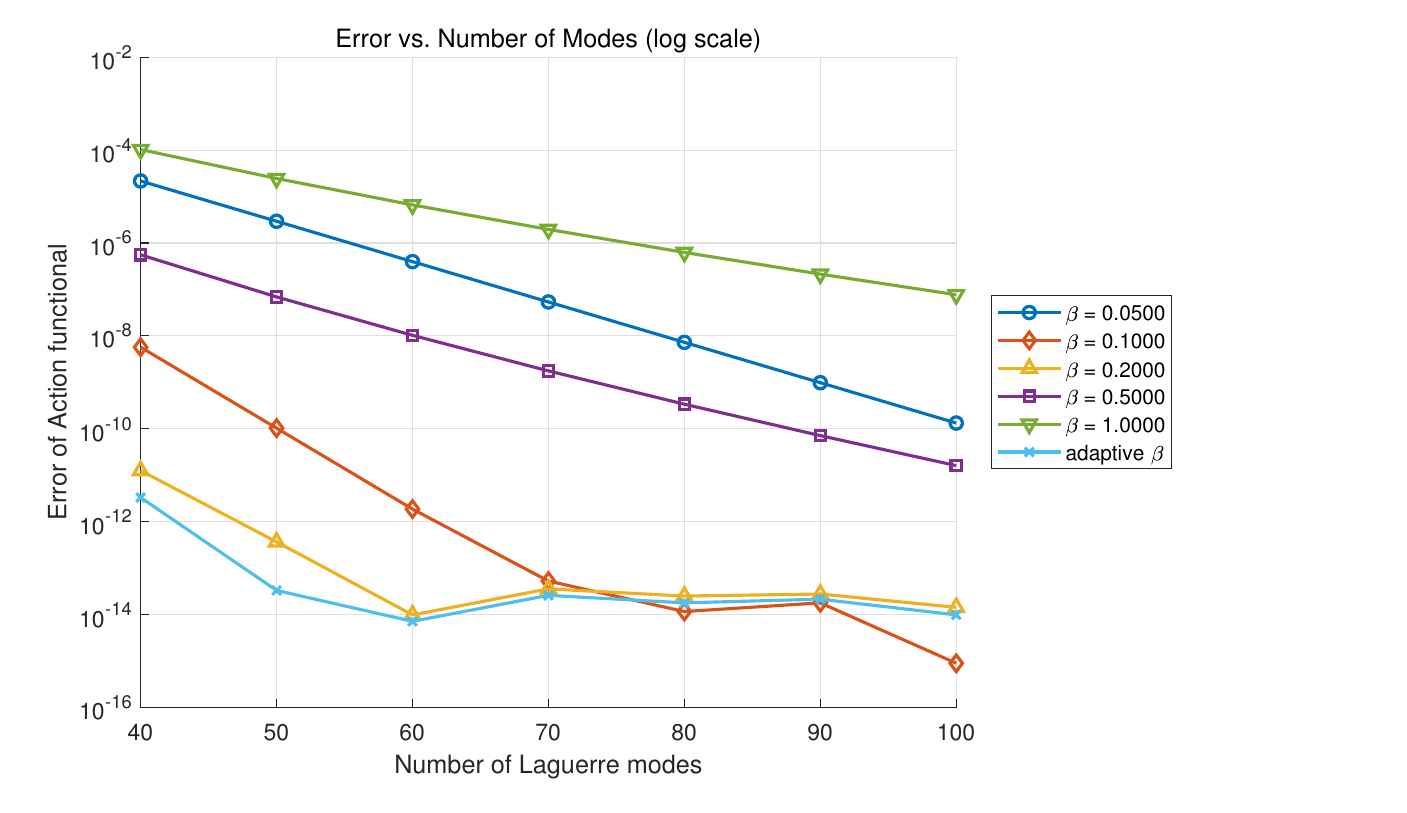}
\caption{The LMAM convergence of action functional for nonlinear gradient example~\eqref{eq:ex6.2-nonlinear-grad-system}.}
\label{fignonlinear} 
\end{figure}

From Figure~\ref{fignonlinear}, we observe again that the error of the action functional decreases geometrically fast as the number of Laguerre modes increases before reaching machine accuracy. And the adaptive $\beta$ strategy consistently achieves the smallest errors 
within machine accuracy over the tested range of $N$.

\subsection{LMAM for Allen--Cahn equation}
We next apply the proposed LMAM to a spatially extended stochastic partial differential
equation.
This example serves as an intermediate benchmark between low-dimensional ODE systems and
high-dimensional fluid models, and is particularly suitable for assessing the performance of
LMAM in infinite-time transition problems arising from gradient flows.
We consider the one-dimensional Allen--Cahn equation (\cite{allen1979microscopic})
\begin{equation}
\partial_t u - \epsilon \Delta u + \frac{1}{\epsilon}(u^3 - u) = 0,
\qquad x \in [-1,1], \; t>0,
\label{eq:allen_cahn}
\end{equation}
subject to homogeneous Dirichlet boundary conditions
\begin{equation}
    \label{eq:Allen-Chan-BC}
u(-1,t)=u(1,t)=0,
\end{equation}
and an initial condition $u(x,0)=u_0(x)$.
This equation can be interpreted as the $L^2$-gradient flow of the Ginzburg--Landau energy
\begin{equation}
E(u)
=
\frac12
\int_{-1}^{1}
\bigl(
\epsilon |\nabla u|^2
+
\frac{1}{\epsilon}(u^2-1)^2
\bigr)\,dx.
\label{eq:ac_energy}
\end{equation}

When subjected to small stochastic perturbations, the Allen--Cahn equation exhibits rare transitions
between metastable states corresponding to local minima of the energy functional.

We apply LMAM to the Allen--Cahn equation. The spatial variable is discretized by Legnedre-spectral method~\cite{Shen1994EfficientSpectralGalerkin,WangYu2019EnergystableSecondorder}. The convergences of the action are shown in Figure \ref{figure:action2ACMy}, for which we see that adaptive scaling algorithm obtains better results than using fixed scaling factor.
\begin{figure}
\centering
\includegraphics[width=0.6\textwidth]{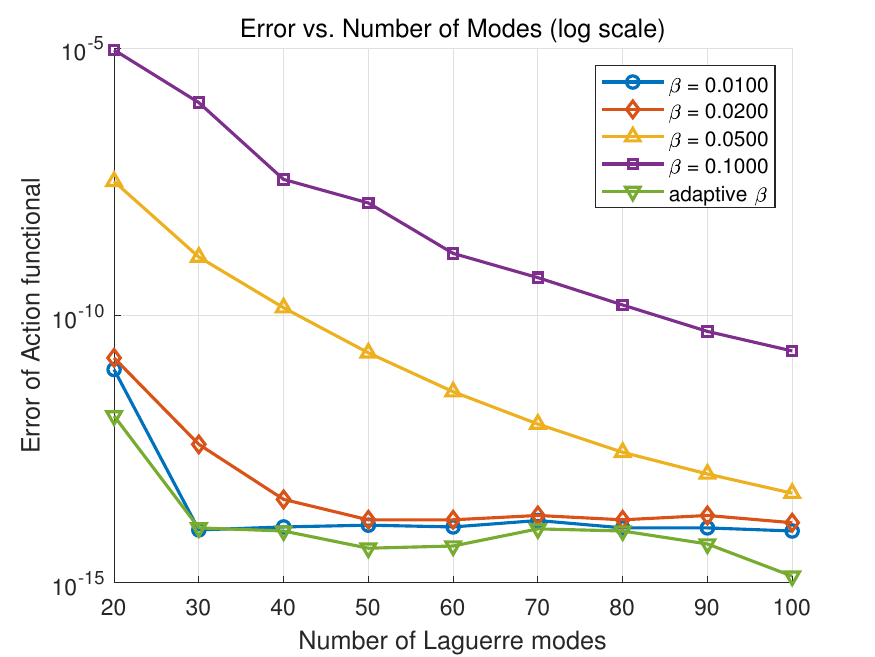}
\caption{The LMAM convergence of action functional for the Allen-Cahn problem \eqref{eq:allen_cahn}-\eqref{eq:Allen-Chan-BC}.}
\label{figure:action2ACMy}
\end{figure}

To further illustrate the structure of the quasi-potential, we project the dynamics onto the first two
dominant eigenmodes of the linearized Allen--Cahn operator around the stable equilibrium. Figure~\ref{figure:action2AC} visualizes the quasi-potential in this reduced coordinate system.
The resulting energy landscape reveals an asymmetric structure along the first eigendirection, denoted by $x_1$, reflecting
the nontrivial interaction between spatial modes induced by the nonlinear term, while the second eigen-direction, denoted by $x_2$ exhibits a more symmetric profile.
\begin{figure}
\centering
\includegraphics[width=0.48\textwidth]{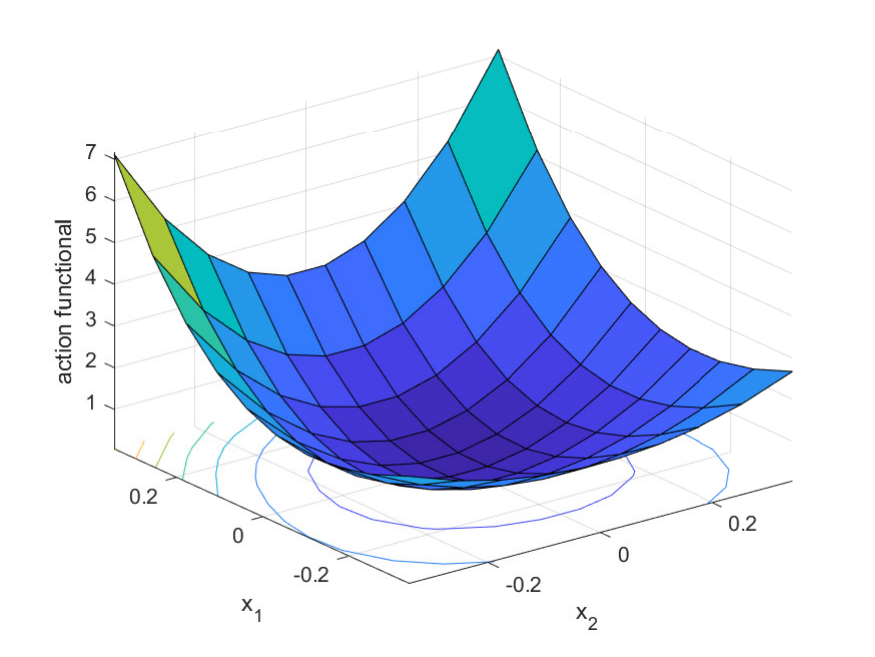}
\hfill
\includegraphics[width=0.48\textwidth]{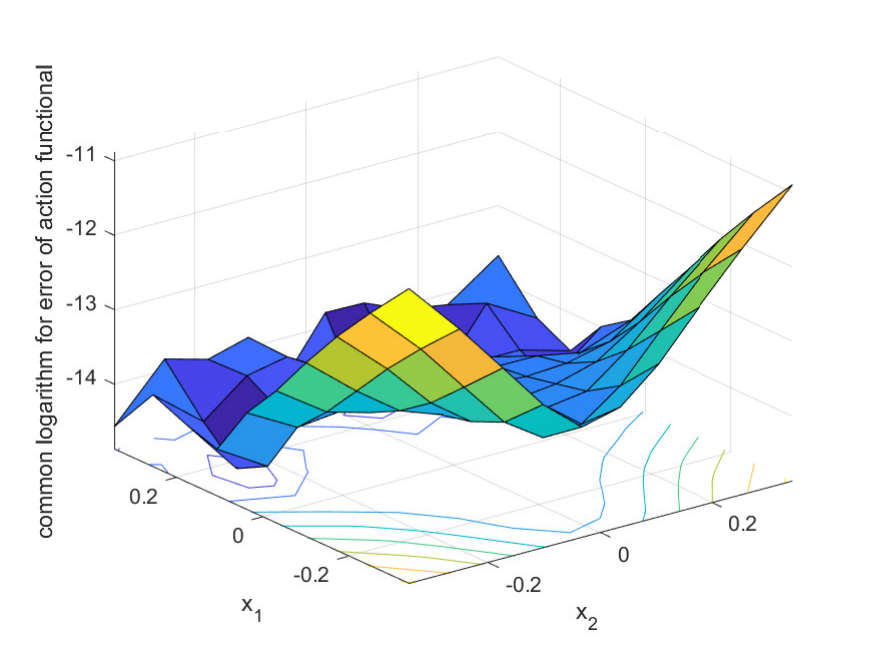}
\caption{The action functional and its error $\epsilon = 0.1$: left: the action functional, right: the error of action functional.}
\label{figure:action2AC}
\end{figure}

\subsection{LMAM for 2-D Navier--Stokes equation}
We finally apply the proposed LMAM to the two-dimensional incompressible Navier--Stokes
equations (\cite{temam2024navier}).
This example constitutes a stringent test for minimum action methods, as it involves a
high-dimensional, non-gradient dynamical system arising from a fundamental model in fluid
mechanics.

We consider the two-dimensional incompressible Navier--Stokes equations defined on the channel
domain $D = [0,L_x] \times [-1,1]$, given by
\begin{equation}
\left\{
\begin{aligned}
\frac{\partial u}{\partial t} + (u_{\mathrm{tot}}\cdot\nabla)u_{\mathrm{tot}}
&= -\nabla p + \frac{1}{\mathrm{Re}} \Delta u, \\
\nabla\cdot u &= 0.
\end{aligned}
\right.
\label{eq:ns}
\end{equation}
where $u_{\mathrm{tot}} = u_b + u$ denotes the total velocity, $u_b$ is a prescribed base flow,
$p$ is the pressure, and $\mathrm{Re}$ is the Reynolds number.
 We consider the Dirichlet boundary condition for $y-$direction
 and periodic boundary condition for $x-$direction:
 \begin{equation} \label{eq:ns-bc}
    \mathbf{u}|_{y=\pm1}=u_{\pm},\ x\in [0,L_x], 
    \qquad
    \mathbf{u}|_{x=0}=\mathbf{u} |_{x=L_x},\ y\in [-1,1].
 \end{equation}
To apply LMAM, we follow the dynamic-solver-consistent framework proposed in
\cite{wan_dynamic_2017, moser1983spectral} to discretize the Navier--Stokes equations in space.
Specifically, to eliminate the pressure variable and enforce incompressibility at the discrete level,
we adopt a divergence-free formulation of the Navier--Stokes equations.
\begin{equation}
	V =
	\bigl\{
	\mathbf{u}\in L^{2}(D)\;
    \mid
    \;
	\nabla\cdot\mathbf{u}=0,\;
	\mathbf{u}|_{y=\pm1}=0,\;
	\mathbf{u}|_{x=0}=\mathbf{u}|_{x=L_x}
	\bigr\},
	\label{space}
\end{equation}
which incorporates the incompressibility constraint and the boundary conditions directly.

Testing the Navier--Stokes equations with $\mathbf{v}\in V$ yields the weak formulation
\begin{equation}
	\left\langle \frac{\partial\mathbf{u}}{\partial t},\mathbf{v}\right\rangle _{D}+\left\langle (\mathbf{u}_{tot,h}\cdot\nabla)\mathbf{u}_{tot,h},\mathbf{v}\right\rangle _{D}=-\left\langle \nabla p,\mathbf{v}\right\rangle_D +\frac{1}{\mathrm{Re}}\left\langle \Delta\mathbf{u},\mathbf{v}\right\rangle _{D}+\left\langle \mathbf{f},\mathbf{v}\right\rangle_D, \quad \mathbf{v} \in V. \label{eq:inn}
\end{equation}

Since $\mathbf{v}$ is divergence-free, the pressure term vanishes,
\begin{equation}
	\left\langle \nabla p,\mathbf{v}\right\rangle_D
	=
	\left\langle p,\nabla\cdot\mathbf{v}\right\rangle_D
	=
	0,
	\label{eq:div}
\end{equation}
and we obtain the pressure-free weak formulation
\begin{equation}
	\left\langle \frac{\partial\mathbf{u}}{\partial t},\mathbf{v}\right\rangle _{D}
	+
	\left\langle (\mathbf{u}_{\mathrm{tot}}\cdot\nabla)\mathbf{u}_{\mathrm{tot}},\mathbf{v}\right\rangle _{D}
	+
	\frac{1}{\mathrm{Re}}\left\langle \nabla\mathbf{u},\nabla\mathbf{v}\right\rangle _{D}
	=
	\left\langle \mathbf{f},\mathbf{v}\right\rangle_D,
	\quad \mathbf{v}\in V.
	\label{eq:divNS}
\end{equation}
Following the dynamic-solver-consistent approach of   \cite{wan_dynamic_2017}, we construct a finite-dimensional divergence-free approximation space by combining Fourier modes in the periodic $x$-direction with polynomial basis functions in the wall-normal $y$-direction.
Specifically, for each Fourier wave number $k$, we define
\[
e_k(x)=\exp\Bigl(i\frac{2\pi k}{L_x}x\Bigr),
\]
which reflects the periodic boundary condition in the streamwise direction. 
To enforce both the no-slip boundary conditions at $y=\pm1$ and the divergence-free constraint, we introduce vector-valued wall-normal basis functions $\overrightarrow{\phi}^{\,k}_m(y)$.
Let $P_m(y)$ denote the Legendre polynomial of degree $m$,
$\varphi_m(y)= P_{m}(y)-P_{m+2}(y)$.
The wall-normal basis functions are defined as follows.

For $k=0$,
\begin{equation}
\overrightarrow{\phi}_{m}^{0}(y)
=
\begin{pmatrix}
\varphi_m(y) \\[1pt]
0
\end{pmatrix},
\quad m\ge 0.
\label{base-y-0}
\end{equation}

For $k\neq 0$,
\begin{equation}
\overrightarrow{\phi}_{m}^{k}(y)
=
\begin{pmatrix}
\varphi_m(y)\\[6pt]
\dfrac{i2\pi k}{L_x}\!\left(\dfrac{\varphi_{m-1}(y)}{2m+1}
-\dfrac{\varphi_{m+1}(y)}{2m+5}\right)
\end{pmatrix},
\quad m\ge 1.
\label{base-y-1}
\end{equation}

Using these wall-normal functions, we define the divergence-free basis functions
\[
\mathbf{v}_{km}(x,y)=e_k(x)\,\overrightarrow{\phi}^{\,k}_m(y).
\]
The resulting basis functions satisfy the periodic boundary condition in $x$, the no-slip boundary conditions at $y=\pm1$, and the incompressibility constraint.

We then define the finite-dimensional approximation space by
\begin{equation}
	V_{h}
	=
	\Bigl\{
	\mathbf{v}\in L^{2}(D)\;\Big|\;
	\mathbf{v}
	=
	\sum_{|k|\leq N}\sum_{m=0}^{M}
	\alpha_{km}\,\mathbf{v}_{km}(x,y)
	\Bigr\},
	\label{space-fin}
\end{equation}
where $\{\alpha_{km}\}$ are the modal coefficients. By construction, $V_h\subset V$ is a finite-dimensional divergence-free subspace.

So we have the finite dimensional weak form,
\[
	\Bigl\langle \frac{\partial\mathbf{u}_h}{\partial t},\mathbf{v}_h\Bigr\rangle _{D}+\left\langle (\mathbf{u}_{tot,h}\cdot\nabla)\mathbf{u}_{tot,h},\mathbf{v}_h\right\rangle _{D}+\frac{1}{\mathrm{Re}}\left\langle \nabla\mathbf{u}_h,\nabla\mathbf{v}_h\right\rangle _{D}=\left\langle \mathbf{f},\mathbf{v}_h\right\rangle_D, \quad \mathbf{v}_h\in V_h.
\]
This step amounts to a spatial discretization of the Navier--Stokes equations, which transforms the
PDE system into a high-dimensional system of ordinary differential equations for the modal
coefficients.

We then use Laguerre MAM to compute the MAPs and quasi-potentials for several Reynolds.
Convergence is assessed by monitoring both the $L^2(0,\infty)$ error of the computed MAPs and the
error of the associated action functional.
Figure~\ref{figure:NS100-500} reports the decay of the path error and action error with respect to the number of Laguerre
modes, where geometrical convergences are obtained, and the adaptive scaling algorithm produces the best results.
\begin{figure}
\centering
    \begin{subfigure}[b]{0.48\textwidth}
        \includegraphics[width=\textwidth]{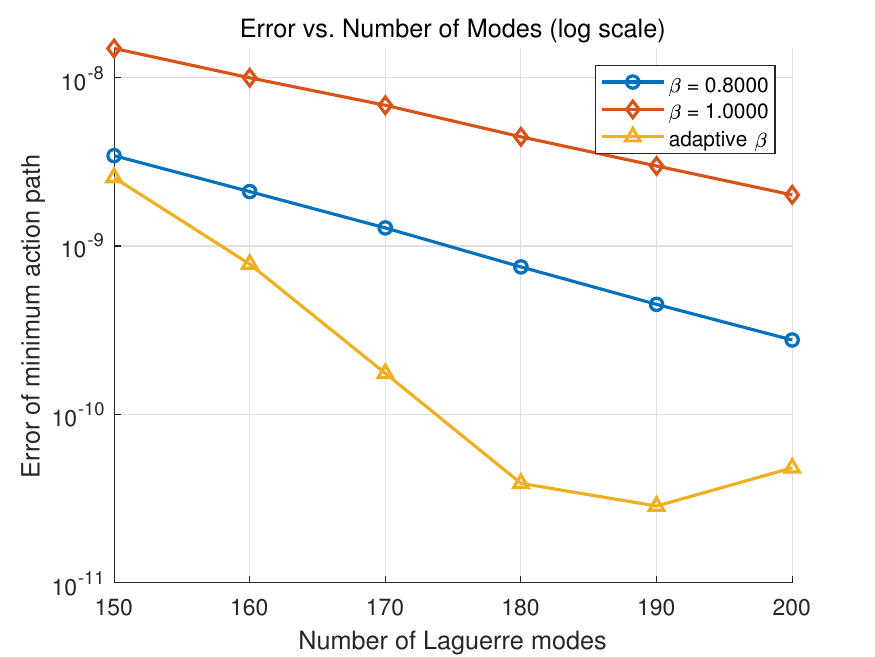}
        \caption{Error of MAPs, $\mathrm{Re}=100$.}
        \label{fig:NS1}
    \end{subfigure}
    \hfill
    \begin{subfigure}[b]{0.48\textwidth}
        \includegraphics[width=\textwidth]{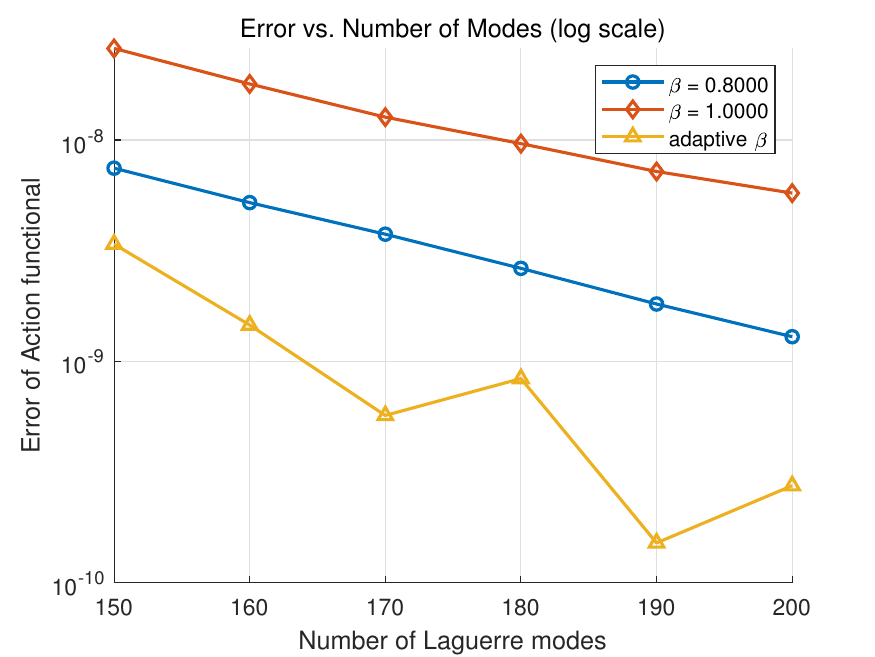}
        \caption{Error of quasi-potential, $\mathrm{Re}=500$.}
        \label{fig:NS2}
    \end{subfigure}    
\caption{The error of MAP and quasi-potential of LMAM for the Navier-Stokes equations.}
\label{figure:NS100-500}
\end{figure}


To gain further insight into the structure of the quasi-potential, we analyze the linearized
Navier--Stokes operator around the base flow and compute its leading eigenmodes. By projecting the dynamics onto the subspace spanned by the two most unstable eigenmodes (presented in Figure~\ref{figure:NS1-4}), we visualize the quasi-potential in this reduced coordinate system in Figure~\ref{figure:NSRe}.
\begin{figure} 
\centering \includegraphics[width=0.48\textwidth]{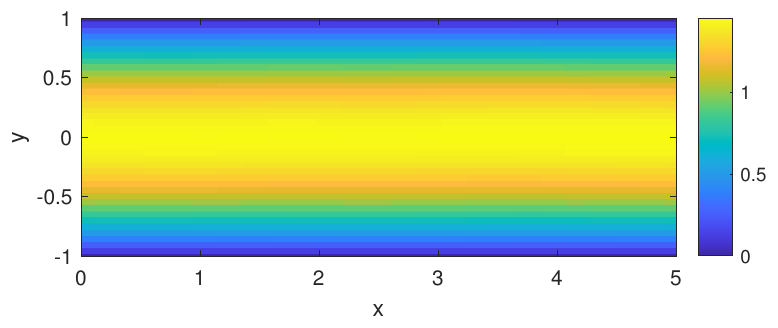}\ \includegraphics[width=0.48\textwidth]{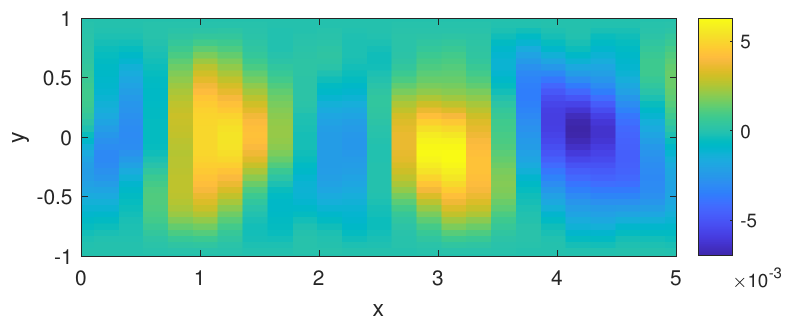}\\ \includegraphics[width=0.48\textwidth]{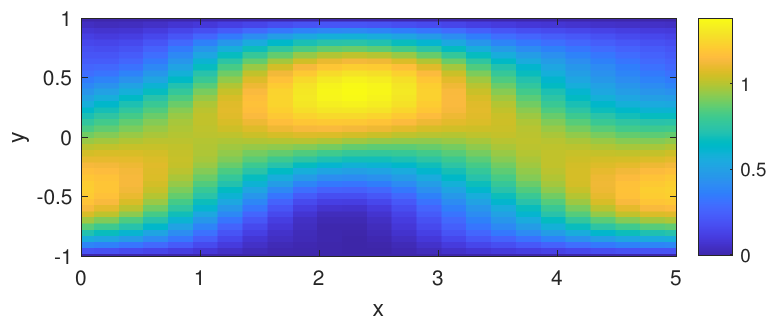}\ \includegraphics[width=0.48\textwidth]{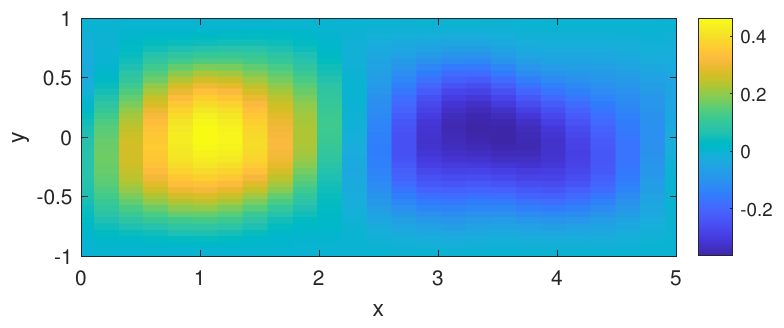} 
\caption{Plots of the two eigemodes $u_1$ and $u_2$ corresponding to the two smallest eigenvalues of the linearized operator around the base solution $u_b$. 
The top row shows the two components of $u_1$, and the bottom row shows the two components of $u_2$. 
The left column represents the $x$-component of the velocity field, while the right column represents the $y$-component.}
\label{figure:NS1-4} 
\end{figure}

\begin{figure}
\centering
\begin{minipage}{0.42\textwidth}
    \centering
    \includegraphics[width=\textwidth]{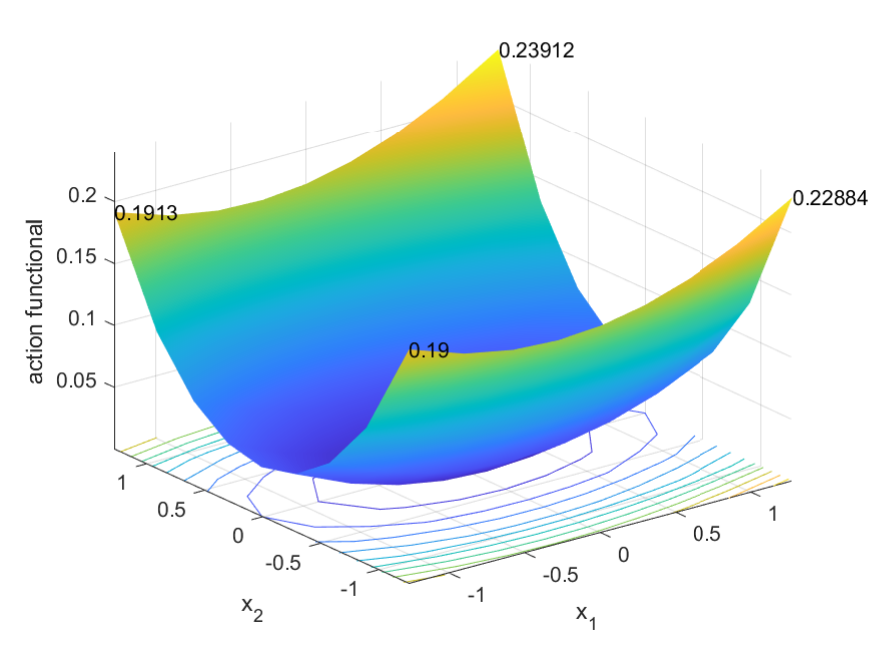}
    \par\vspace{1mm}
    \small (a) $Re=100$
\end{minipage}
\hspace{0.005\textwidth}
\begin{minipage}{0.42\textwidth}
    \centering
    \includegraphics[width=\textwidth]{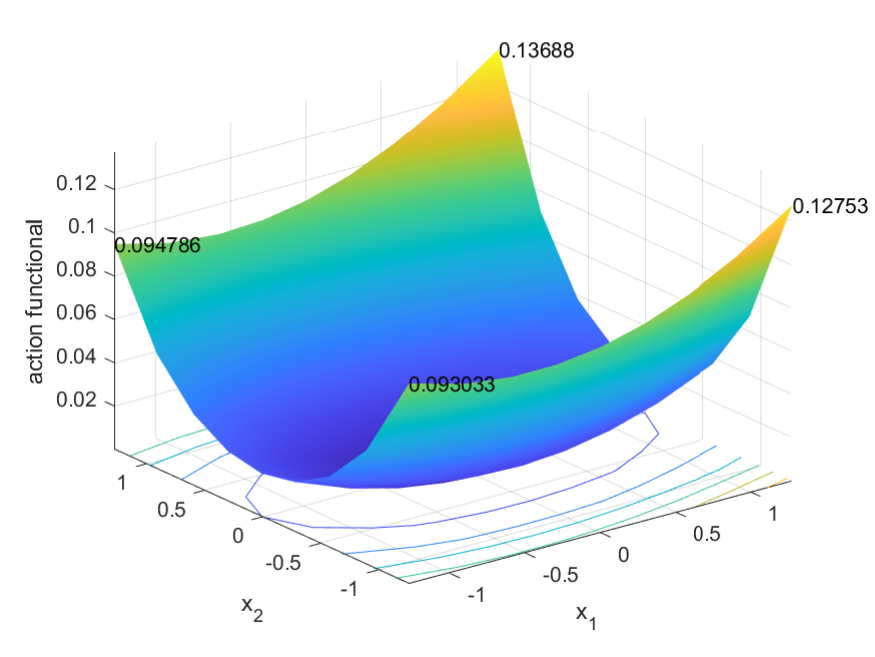}
    \par\vspace{1mm}
    \small (b) $Re=200$
\end{minipage}

\vspace{1mm}

\begin{minipage}{0.42\textwidth}
    \centering
    \includegraphics[width=\textwidth]{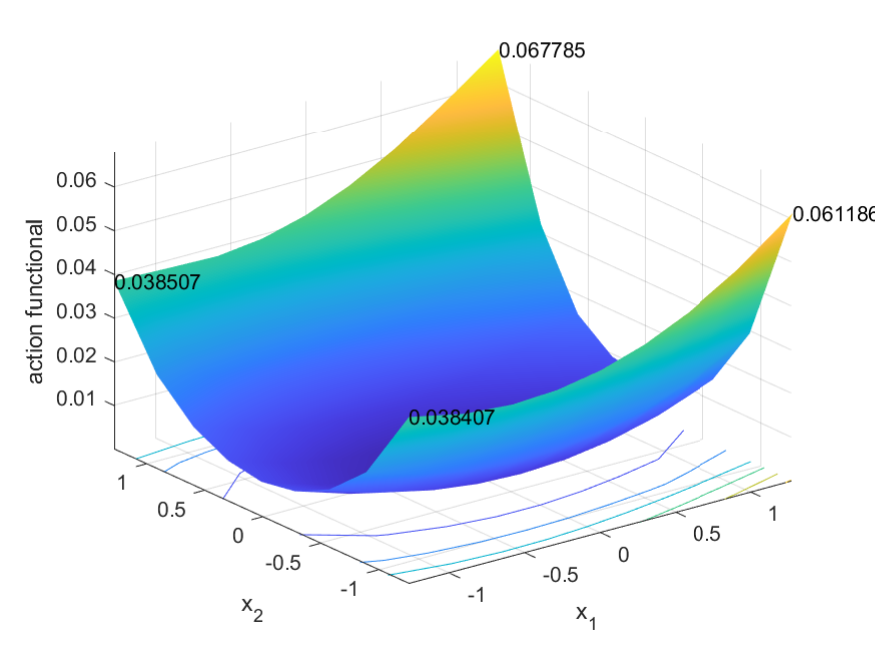}
    \par\vspace{1mm}
    \small (c) $Re=500$
\end{minipage}
\hspace{0.005\textwidth}
\begin{minipage}{0.42\textwidth}
    \centering
    \includegraphics[width=\textwidth]{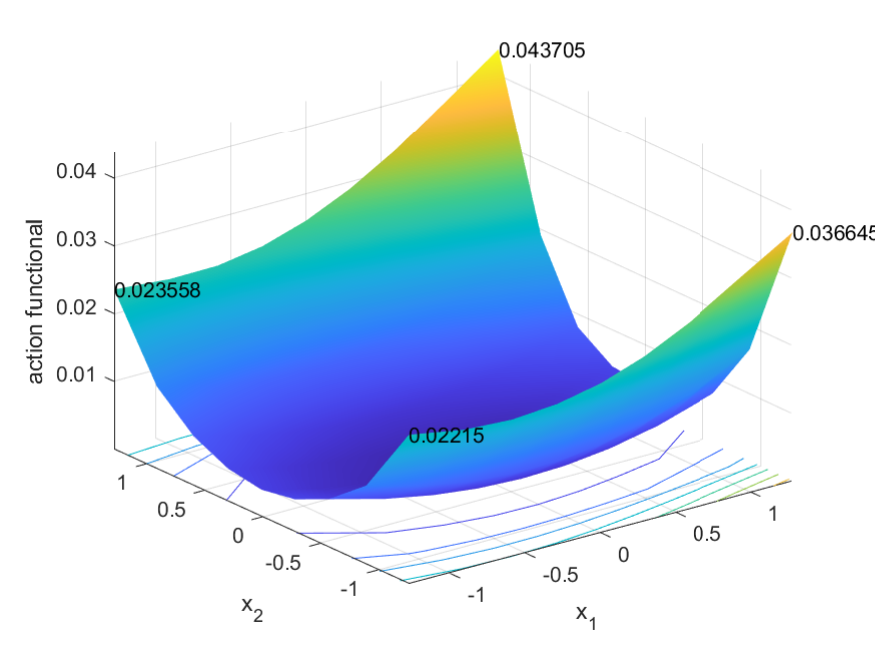}
    \par\vspace{1mm}
    \small (d) $Re=1000$
\end{minipage}
\caption{Visualization of the quasi-potential around the laminar base flow $\mathbf{u}_b$ in the two-dimensional subspace spanned by $\mathbf{u}_1$ and $\mathbf{u}_2$, parameterized by perturbations of the form $\mathbf{u}_b + x_1\mathbf{u}_1 + x_2\mathbf{u}_2$ at different Reynolds number. }
\label{figure:NSRe}
\end{figure}
We can apply the LMAM to the linearized Navier--Stokes equation:
\begin{equation}
	\frac{\partial\mathbf{u}}{\partial t}+(\mathbf{u}_{b}\cdot\nabla)\mathbf{u} + (\mathbf{u}\cdot\nabla)\mathbf{u}_b=-\nabla p+\frac{1}{\mathrm{Re}}\Delta\mathbf{u}\label{eq:linNS'},
\end{equation}
Using the same analysis, we figure out the quasi-potential to the linearized equation in  Figure \ref{figure:NSRelin}.
\begin{figure}
\centering
\begin{minipage}{0.42\textwidth}
    \centering
    \includegraphics[width=\textwidth,trim=8 8 8 8,clip]{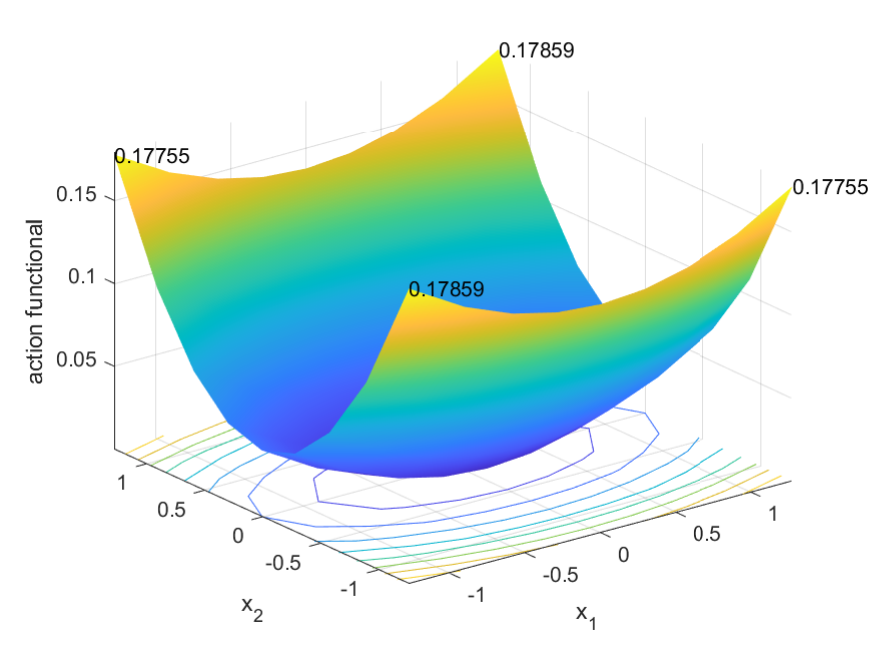}
    \par\vspace{1mm}
    \small (a) $Re=100$
\end{minipage}
\hspace{0.005\textwidth}
\begin{minipage}{0.42\textwidth}
    \centering
    \includegraphics[width=\textwidth,trim=8 8 8 8,clip]{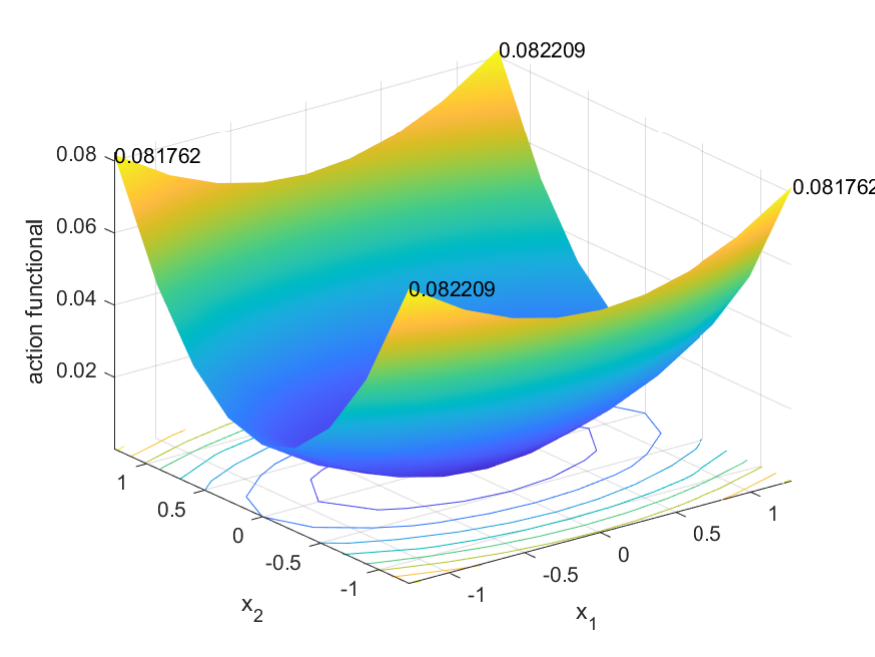}
    \par\vspace{1mm}
    \small (b) $Re=200$
\end{minipage}

\vspace{1mm}

\begin{minipage}{0.42\textwidth}
    \centering
    \includegraphics[width=\textwidth,trim=8 8 8 8,clip]{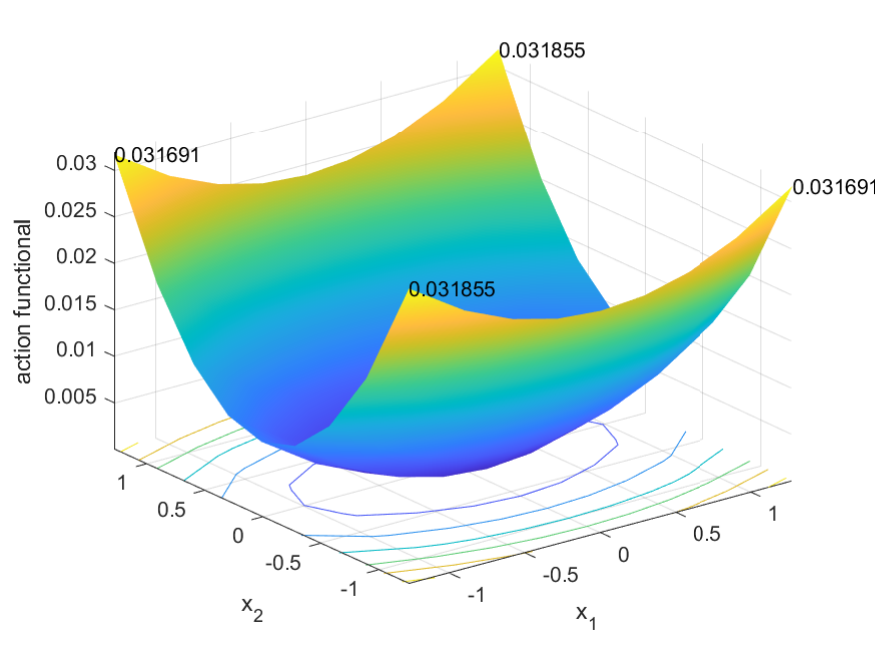}
    \par\vspace{1mm}
    \small (c) $Re=500$
\end{minipage}
\hspace{0.005\textwidth}
\begin{minipage}{0.42\textwidth}
    \centering
    \includegraphics[width=\textwidth,trim=8 8 8 8,clip]{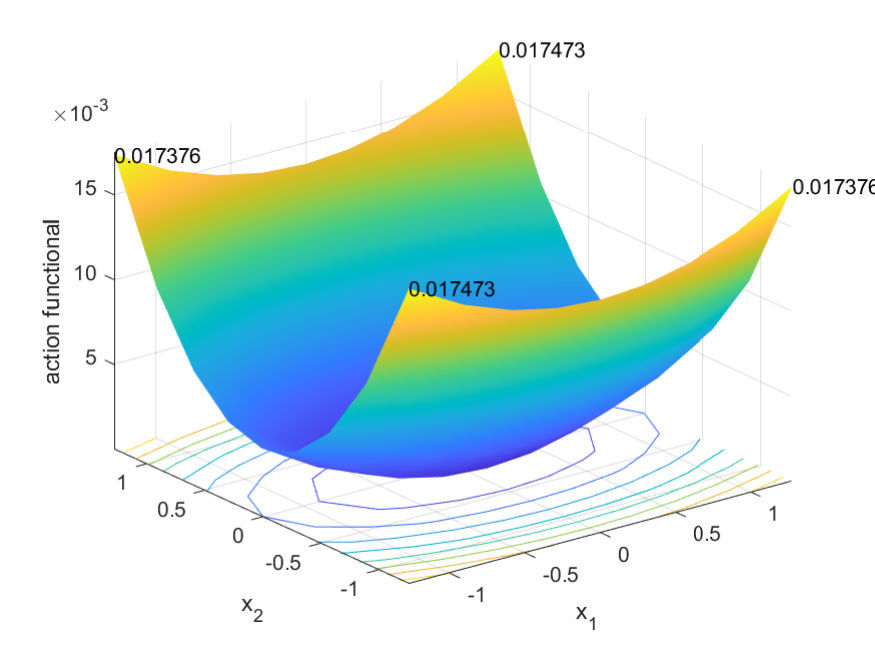}
    \par\vspace{1mm}
    \small (d) $Re=1000$
\end{minipage}
\caption{Visualization of the quasi-potential for the linearized Navier--Stokes equation around the laminar base flow $\mathbf{u}_b$ in the two-dimensional subspace spanned by $\mathbf{u}_1$ and $\mathbf{u}_2$, parameterized by perturbations of the form $\mathbf{u}_b + x_1\mathbf{u}_1 + x_2\mathbf{u}_2$ at different Reynolds number.}
\label{figure:NSRelin}
\end{figure}

From Figure~\ref{figure:NSRelin} and Figure~\ref{figure:NSRe}, we observe that for the linearized case, 
the quasi-potential is almost linearly inverse proportional to Reynolds number, being almost even function of both $x_1$ and $x_2$ coordinates. But, for the non-linear Navier-Stokes equations, even though the quasi-potential decreases when Reynolds number increases, but not in a inverse proportional way. Moreover, the quasi-potential is no longer an even function of $x_1$ and $x_2$, reflecting complicated nonlinear effects.

\section{Conclusion}
\label{conclusion}
We proposed an efficient Laguerre MAM for computing minimum-action paths and quasi-potentials in dynamical systems driven by small noise. By formulating the variational problem over a semi-infinite time interval and adopting Laguerre functions for temporal discretization, the method inherently resolves infinite-time transition trajectories without the need for artificial time truncation. For linear systems, we establish geometric convergence and show that the convergence rate can be effectively tuned by the time-scaling factor, which lays a solid theoretical foundation for the proposed adaptive scaling strategy. 

Extensive numerical experiments are conducted on both low-dimensional examples and spatially extended models, including the Allen--Cahn equation and the two-dimensional Navier--Stokes equations. These experiments comprehensively demonstrate the robustness and accuracy of the proposed Laguerre MAM, where adaptive time scaling significantly enhances convergence and reduces sensitivity to parameter selection.

Several directions warrant further investigation in future research. First, the current formulation is tailored to transitions involving a single metastable state; extending the framework to multi-stage transitions that traverse multiple critical points (e.g., through domain decomposition or  coupling multiple semi-infinite Laguerre expansions) would broaden its applicability. Second, for high-dimensional PDE discretizations the resulting optimization problem can be severely ill-conditioned; developing scalable preconditioners and fast solvers that leverage the inherent spectral structure (including fully diagonalized variants) is crucial for enabling large-scale practical applications. Third, it would be valuable to deepen the theoretical understanding of genuinely nonlinear systems, including the derivation of sharper a priori/a posteriori error estimates and rigorous guarantees for the adaptive scaling update rule. 


\section*{Acknowledgement}
We would like to thank Prof. Jie Shen and Huiyuan Li for helpful discussions on
Laguerre spectral methods. This work was supported by the National Natural Science Foundation of China under Grant No. 12494543, 92370205, 12171467,
Strategic Priority Research Program of Chinese Academy of Sciences under grant XDA0480504. 
The computations were partially performed on the LSSC4 PC cluster of State Key Laboratory of Scientific and Engineering Computing.

\bibliographystyle{cas-model2-names}
\bibliography{refs}
\appendix

\end{document}